\documentclass[11pt]{amsart}

\usepackage{color}

\usepackage{mathptmx} 
\usepackage[scaled=0.90]{helvet} 
\usepackage{courier} 
\normalfont
\usepackage[T1]{fontenc}

\usepackage{bookmark}

\usepackage{hyperref}
\hypersetup{bookmarksdepth=3}

\usepackage{geometry}
\usepackage{amsmath,amssymb}
\usepackage{
mathrsfs}
\usepackage{esint}

\usepackage{tensor}

\usepackage[hide]{ed} 

\allowdisplaybreaks[1]

    \normalbaselines                          %

\newtheorem{thm}{Theorem}[section]
\newtheorem{lm}[thm]{Lemma}

\theoremstyle{definition}
\newtheorem{df}[thm]{Definition}
\newtheorem*{df*}{Definition}

\theoremstyle{remark}
\newtheorem{rem}[thm]{Remark}
\newtheorem*{rem*}{Remark}

\numberwithin{equation}{section}

\newcommand{\ci}[1]{_{ {}_{\scriptstyle #1}}}
\newcommand{\ti}[1]{_{\scriptstyle \text{\rm #1}}}

\newcommand{\1}{\mathbf{1}}

\newcommand{\supp}{\operatorname{supp}}

\newcommand{\la}{\lambda}

\newcommand{\cB}{\mathcal{B}}
\newcommand{\cG}{\mathcal{G}}

\newcommand{\al}{\alpha}

\newcommand{\cz}{Calder\'{o}n--Zygmund\ }
\newcommand{\C}{\mathbb{C}}
\newcommand{\Z}{\mathbb{Z}}

\newcommand{\dist}{\operatorname{dist}}

\newcommand{\Om}{\Omega}
\newcommand{\om}{\omega}
\newcommand{\bP}{\mathbf{P}}

\newcommand{\sA}{\mathfrak{A}}

\newcommand{\E}{\mathbb{E}}
\newcommand{\bE}{\mathbf{E}}

\newcommand{\R}{\mathbb{R}}

\newcommand{\QQ}{[w]_{A_2}}

\newcommand{\ch}{\operatorname{ch}}
\newcommand{\ran}{\operatorname{Ran}}
\DeclareMathOperator*{\esssup}{ess\,sup}

\newcommand{\wt}{\widetilde}
\newcommand{\La}{\langle}
\newcommand{\Ra}{\rangle}
\newcommand{\N}{\mathbb{N}}
\newcommand{\cD}{\mathscr{D}}
\newcommand{\ccA}{\mathscr{A}}
\newcommand{\cQ}{\mathcal{Q}}
\newcommand{\cP}{\mathcal{P}}
\newcommand{\cA}{\mathcal{A}}
\newcommand{\cE}{\mathcal{E}}

\newcommand{\gw}{{\text{good},\omega}}

\def\cyr{\fontencoding{OT2}\fontfamily{wncyr}\selectfont}
\DeclareTextFontCommand{\textcyr}{\cyr}
\newcommand{\sha}[0]{\ensuremath{\mathbb{S}
}}

{\end{list}}      

\renewcommand{\labelenumi}{(\roman{enumi})}

\newcounter{vremennyj}

\newcommand\cond[1]{\setcounter{vremennyj}{\theenumi}\setcounter{enumi}{#1}\labelenumi\setcounter{enumi}{\thevremennyj}}

\newcommand{\fdot}{\,\cdot\,}

\begin{document}

\title[$A_2$ conjecture]{Sharp weighted estimates for dyadic shifts and the $A_2$ conjecture}
\author{Tuomas Hyt\"onen}
\address{Department of Mathematics and Statistics, P.O.B. 68, FI-00014 University of Helsinki, Finland}
\email{tuomas.hytonen@helsinki.fi}
\thanks{Work of T.~Hyt\"onen is supported by the Academy of Finland under grants 130166, 133264 and 218148.}

\author{Carlos P\'erez}
\address{Department of mathematics, Universidad de Sevilla, Sevilla, Spain}
\email{carlosperez@us.es}
\thanks{Work of C.~P\'erez is  supported by the Spanish Research Council grant}
\author{Sergei Treil}
\thanks{Work of S.~Treil is supported  by the National Science Foundation under the grant  DMS-0800876. 
}
\address{Dept. of Mathematics, Brown University,   
151 Thayer
Str./Box 1917,      
 Providence, RI  02912, USA }
\email{treil@math.brown.edu}
\urladdr{http://www.math.brown.edu/\~{}treil}
\author{Alexander Volberg}
\thanks{Work of A.~Volberg is supported  by the National Science Foundation under the grant  DMS-0758552. 
}
\address{Department of Mathematics, Michigan State University, East
Lansing, MI 48824, USA}
\email{volberg@math.msu.edu}
\urladdr{http://sashavolberg.wordpress.com}
\makeatletter
\@namedef{subjclassname@2010}{
  \textup{2010} Mathematics Subject Classification}
\makeatother

\subjclass[2010]{42B20, 42B35, 47A30}



%
%

\keywords{\cz operators, $A_2$ weights, Carleson embedding theorem, Corona decomposition, stopping time,
   nonhomogeneous Harmonic Analysis.}
\date{}

\begin{abstract}
We give a  self-contained proof of the $A_2$ conjecture, which claims that the norm of \emph{any} \cz operator is bounded by the first degree of the $A_2$ norm of the weight. The original proof of this result by the first author relied on a subtle and rather difficult  reduction to a testing condition by the last three authors. Here we replace this reduction by a new weighted norm bound for dyadic shifts --- linear in the $A_2$ norm of the weight and quadratic in the complexity of the shift ---,  which is based on a new quantitative two-weight inequality for the shifts. These sharp one- and two-weight bounds for dyadic shifts are the main new results of this paper. They are obtained by rethinking the corresponding previous results of Lacey--Petermichl--Reguera and Nazarov--Treil--Volberg. 
To complete the proof of the $A_2$ conjecture, we also provide a simple variant of the representation, already in the original proof, of an arbitrary \cz operator as an average of random  dyadic shifts and random dyadic paraproducts. This method of the representation amounts to the refinement of the techniques from nonhomogeneous Harmonic Analysis. 
\end{abstract}

\maketitle

\section{Introduction}
\label{intro}


A \emph{\cz operator} in $\R^d$ is an integral operator, bounded in $L^2$ and with kernel $K$ satisfying the following growth and  smoothness conditions 
\begin{enumerate}
\item    $\displaystyle |K(x, y)|   \le \frac{C\ti{cz}}{|x-y|^d}$ for all $x, y\in \R^d$, $x\ne y$. 

\item There exists $\alpha>0$ such that 
\[
 |K(x, y) - K(x', y) | + |K(y, x) - K(y, x')| \le C\ti{cz} \frac{|x-x'|^{\alpha}}{|x-y|^{d+\alpha}}
\] 
for all $x, x', y\in\R^d$ such that $|x-x'|<|x-y|/2$.     
\end{enumerate}

It is well known that a \cz operator is bounded in the weighted space $L^2(w)$ if (and for many \cz operators only if) the weight $w$ satisfies the famous Muckenhoupt $A_2$ condition 
\begin{equation}
\label{A2cond}
\sup_{Q} \left(|Q|^{-1} \int_Q w dx \right) \left(|Q|^{-1} \int_Q w^{-1} dx \right) =: [w]\ci{A_2} <\infty.
\end{equation}
The quantity $[w]\ci{A_2}$ is called the \emph{Muckenhoupt norm} of the weight $w$ (although it is definitely not a norm). 

It has been an old problem to describe how the norm of a \cz operator in the weighted space $L^2(w)$ depends on the Muckenhoupt norm $[w]\ci{A_2}$ of $w$. A conjecture was that for a fixed \cz operator $T$ its norm is bounded by $C\cdot [w]\ci{A_2}$, where the constant $C$ depends on the operator $T$ (but not on the weight $w$). Simple counterexamples demonstrate that for the classical operators like Hilbert Transform or Riesz Transform, a better estimate than $C\cdot [w]\ci{A_2}$ is not possible. 

This linear (in $[w]\ci{A_2}$) estimate of the norm has become known as the \emph{$A_2$ conjecture}. 

For the maximal function, the estimate  $C\cdot [w]\ci{A_2}$ was proved by S.~Buckley \cite{Buck1}: he also proved that this estimate is optimal for the maximal function. 
The first result for a singular ``integral'' operator was due to J.~Wittwer \cite{Wit}, who proved the  $A_2$ conjecture for the Haar mutipliers. The same result for Beurling--Ahlfors Transform (convolution with $\pi^{-1} z^{-2}$ in $\C$) was obtained first by Petermichl--Volberg \cite{PetmV} by using the combination of Bellman function technique and the heat extension,  and later by Dragicevic--Volberg \cite{Dr-Volb-A2_Beurling_2003} via the representation of the Beurling--Ahlfors  Transform as an average of Haar multipliers over all dyadic lattices.  

This result was used in \cite{PetmV} to answer positively an important question in the theory of quasiconformal maps, see \cite{AIS}, about whether a weakly quasiregular map is quasiregular (or equivalently whether there is a self-improvement of a solution of the Beltrami equation in the case of critical exponent).

Then S.~Petermichl \cite{Petm1} proved the $A_2$ conjecture for the Hilbert transform, again using the representation of the Hilbert Transform as an average of copies of a simple dyadic operator (the so-called dyadic, or Haar, shift of complexity $1$).  

We should mention here an earlier paper by R.~Fefferman and J.~Pipher  \cite{FP}, where a linear estimate in terms of stronger $A_1$ norm of the weight $w$ was obtained for he Hilbert Transform. This result found its application in geometric questions pertinent to multi-parameter Harmonic Analysis, in particular for singular operators on Heisenberg group. The result in \cite{Petm1} is a considerable strengthening of Fefferman--Pipher's theorem.

A recent paper \cite{LPR} by M.~Lacey, S.~Petermichl and M.~Reguera established the $A_2$ conjecture for general dyadic shifts. 
Another proof of the linear bound for dyadic shifts was obtained in Cruz-Uribe--Martell--P\'erez \cite{CUMP1}, \cite{CUMP2} in a very beautiful and concise approach based on a remarkable ``formula'' by  Lerner \cite{Ler1}.
Thus, the conjecture was proved  for all operators which can be represented by taking for each dyadic grid a sum of finitely many dyadic shifts of uniformly bounded complexity (see definition below) and taking the average over all grids.

In particular, as it was shown by A.~Vagharshakyan \cite{Vagh_HaarShift_2009}, any convolution \cz operator on the real line $\R$  with sufficiently smooth kernel can be obtained by averaging copies of just one Haar shift, so the $A_2$ conjecture holds for such operators.  

Note that estimates of the norms of the dyadic shifts  obtained in  \cite{LPR} and in cite{CUMP1}, \cite{CUMP2} grew exponentially in the complexity of the shift, so it was only possible to estimate the \cz operators obtained by averaging of finitely many such shifts.

Using  linear estimates for the dyadic shifts and  a special decomposition (in the form proposed by Xiang \cite{X}) of a \cz operator 
 Hyt\"onen--Lacey--Reguera--Sawyer--Vagharshakyan--Uriarte-Tuero in \cite{HLRSUTV}  proved $A_2$ conjecture for all \cz operator with sufficiently smooth kernels (the smoothness was dependent on the dimension in \cite{HLRSUTV}). However, the problem for general \cz operator  required (as we shall see) some probabilistic ideas rooted in non-homogeneous Harmonic Analysis \cite{NTV1},  \cite{NTV5} (see also the lecture notes \cite{V}).

For general \cz operators, the last three authors \cite{PTV1} reduced the $A_2$ conjecture to a weak type estimate by establishing the inequality 
\[
\|T\|\ci{L^2(w)\rightarrow L^2(w)} 
\le C\,\left(\QQ+ \|T\|\ci{L^2(w)\rightarrow L^{2,\infty}(w)}+ \|T'\|\ci{L^2(w^{-1})\rightarrow L^{2,\infty}(w^{-1})} \right).
\] 
In \cite{PTV1} it is also shown that  $A_2$ conjecture is equivalent to  getting the linear in $[w]_{A_2}$ estimate on simplest test functions (this is a $T(1)$ theorem in the presence of weight).
Using this result of P\'erez--Treil--Volberg and the technique developed in \cite{LPR} the first author in \cite{H} was able to prove the $A_2$ conjecture for general \cz operators, i.e., the following theorem:

\begin{thm}[\cite{H}]
\label{A2}
Let $T$ be a \cz operator and $w$ be an $A_2$ weight. Then 
\[
\|T f\|\ci{L^2(w)} \le C\cdot [w]_{A_2} \| f\|\ci{L^2(w)}, 
\]
where the constant $C$ depends only on the dimension $d$, the parameters $C\ti{cz}$, $\alpha$ of the \cz operator and its norm in the non-weighted $L^2$. 
\end{thm}

A crucial new element in \cite{H} was a clever averaging trick, allowing one to get rid of the so called \emph{bad} cubes and thus represent an arbitrary \cz operator as a weighted average of (infinitely many) dyadic shifts. This averaging trick was a development of the bootstrapping argument used by Nazarov--Treil--Volberg \cite{NTV5}, where they exploited the fact that the bad part of a function can be made arbitrarily small. Using the original Nazarov--Treil--Volberg averaging trick would add an extra factor depending on $[w]\ci{A_2}$ to the estimate, so a new idea was necessary. A new observation in \cite{H} was that as soon as the probability of a ``bad'' cube is less than $1$, it is possible to completely ignore the bad cubes (at least in the situation where they cause troubles).   

The preprint \cite{H}, which itself is neither short or very simple, relies of a rather technically involved preprint \cite{PTV1}. Thus the necessity of a simpler, direct proof, not using the reduction to the weak type estimates seems pretty evident.  

Such a direct proof of Theorem~\ref{A2} is presented in this paper; moreover,  we obtain new results on the dyadic shifts into which the \cz operator $T$ is decomposed. Indeed, the reduction of the $A_2$ conjecture to a testing condition, which in \cite{PTV1} was made on the level of the \cz operator $T$, is here performed on the more elementary level of the dyadic shifts in the representation of $T$. The possibility of such a simplification in the proof of the $A_2$ conjecture was suggested in \cite{H}, Sec.~8.A, and here we carry out this program in detail.

The main components of the proof are as follows:

\begin{enumerate}
    \item An averaging trick, which is a version of the one from \cite{H} (unlike \cite{H} we do not need {\it good} shifts here, and this simplifies the matter). This trick allows us not to worry about ``bad'' cubes and represent a general \cz operator as a weighted average of dyadic shifts with the weights decaying exponentially in the complexity of the shifts. 
    
    \item Sharp estimates, with all the constants written down, in the two weight $T(1)$ theorem from \cite{NTV6} in the setting of dyadic shifts (Theorem~\ref{sh2w}). Note, that while most of the necessary estimates were done in \cite{NTV6}, a formal application of the result from \cite{NTV6} would give an exponential (in complexity) growth of the norm. 
    
    To get the polynomial (in complexity) growth, one needs some non-trivial modifications. 
    For the convenience of the reader we present the complete proof, not only the modifications: only describing modifications and referring the reader to the proof in \cite{NTV6} would make the paper unreadable. 
    
    \item A modification of the proof from \cite{LPR}, which gives polynomial in complexity, instead of exponential, as in \cite{LPR}, bound for the weighted norm of the dyadic shift (Theorem~\ref{t:sharp-shift-A2}).  The main difference compared to \cite{LPR} is a better (linear in complexity instead of exponential) estimate of the (non-weighted) weak $L^1$ norm of a dyadic shift, which was obtained in \cite{H}. 
    
    The rest of the proof essentially follows the construction from \cite{LPR}, keeping track of constants, and clarifying  parts of the proof that were  presented there in a sketchy way.  We note that a variant of such a modification of \cite{LPR} already appeared in \cite{H}, where it was used to verify the required testing conditions for $T$, but not an explicit norm bound for the shifts themselves.
\end{enumerate}

Aside from the new self-contained proof of Theorem~\ref{A2}, the above-mentioned Theorems~\ref{sh2w} and \ref{t:sharp-shift-A2}, giving sharp quantitative two-weight and one-weight bounds for dyadic shifts, are the main new results of this paper.

\section{Dyadic lattices and martingale difference decompositions. Random dyadic lattices}

\subsection{Random dyadic lattices}
\label{s:RDL}
The standard dyadic system in $\R^d$ is
\begin{equation*}
  \cD^{0}:=\bigcup_{k\in\Z}\cD^{0}_k,\qquad
  \cD^{0}_k:=\big\{ 2^{k}\big([0,1)^d+m\big):m\in\Z^d\big\}.
\end{equation*}
For $I\in\cD_k^0$ and a binary sequence $\om=(\om_j)_{j=-\infty}^{\infty}\in(\{0,1\}^d)^{\Z}$, let 
\begin{equation*}
  I\dot+\om:=I+\sum_{j<k}\om_j 2^{j}.
\end{equation*}
Following Nazarov, Treil and Volberg \cite[Section 9.1]{NTV5},  consider general dyadic systems of the form
\begin{equation*}
  \cD=\cD^{\om}:=\{I\dot+\om:I\in\cD^0\}
  =\bigcup_{k\in\Z}\cD^{\om}_k.
\end{equation*}
Given a cube $I=x+[0,\ell)^d$, let
\begin{equation*}
   \ch(I):=\{x+\eta\ell/2+[0,\ell/2)^d:\eta\in\{0,1\}^d\} 
\end{equation*}
denote the collection of dyadic children of $I$. Thus $\cD^{\om}_{k-1}=\bigcup\{\ch(I):I\in\cD^{\om}_k\}$. Note that, in line with \cite{NTV5} but contrary to \cite{H}, we use the ``geometric'' indexing of cubes, where larger $k$ refers to larger cubes, rather than the ``probabilistic'' indexing, where larger $k$ would refer to finer sigma-algebras.

Consider the standard probability measure on $\{0,1\}^d$, which assigns equal probability $2^{-d}$ to every point. Define the measure $\bP$ on $(\{0,1\}^d)^\Z$ as the corresponding product measure. 

\subsection{Martingale difference decompositions and Haar functions}

For a cube $I$ in $\R^d$ let 
\[
\E\ci I f := \left(\fint_I f dx\right) \1\ci I := \left(|I|^{-1}\int_I f dx\right) \1\ci I, \qquad \Delta\ci I :=-\E\ci I + \sum_{J\in \ch(I)} \E\ci{J} . 
\]
It is well known that for an arbitrary dyadic lattice $\cD$ every function $f\in L^2(\R^d)$ admits the orthogonal decomposition 
\[
f= \sum_{I\in\cD} \Delta\ci I f. 
\]

We also need the weighted martingale difference decomposition. Let $\mu$ be a Radon measure on $\R^d$. Define the weigted expectation and martingale differences as
\[
\E^\mu\ci I f := \left((\mu(I))^{-1}\int_I f d\mu\right) \1\ci I, \qquad \Delta^\mu\ci I :=-\E^\mu\ci I + \sum_{J\in \ch(I)} \E^\mu\ci{J};  
\]
for the definiteness we set $\E^\mu\ci I f = 0$ if $\mu(I)=0$. 

For an arbitrary dyadic lattice $\cD$ and $k\in\Z$, any function  $f\in L^2(\mu)$ admits an orthogonal decomposition 
\begin{equation}
\label{mdd-mu}
f = \sum_{I\in\cD: \ell(I) =2^k} \E^\mu\ci I f + \sum_{I\in\cD: \ell(I) \leq 2^k} \Delta^\mu\ci I f
\end{equation}

Given a cube $Q$ in $\R^d$, any function in the martingale difference space $\Delta\ci Q L^2$ is called a Haar function (corresponding to $Q$) and is usually denoted by  $h\ci Q$. Note, that $h\ci Q$ denotes a \emph{generic} Haar function, not any particular one. 

A \emph{generalized} Haar function $h\ci Q$ is a linear combination of a Haar function and $\1\ci Q$. In other words,  a generalized Haar function $h\ci Q$ is constant on the children of $Q$, but unlike the regular Haar function it is not orthogonal to constants. 

Similarly a function $h\in \Delta^\mu\ci Q L^2(\mu)$  is called a weighted Haar function and is denoted as $h^\mu\ci Q$.

\section{Dyadic shifts. A sharp two weight estimate}
\label{DS2w}

\begin{df} An unweighted dyadic paraproduct is an operator $\Pi$ of the form
\[
\Pi f =\sum_{Q\in \cD} (\E\ci Q f )  h\ci Q\,,
\]
where $h\ci Q$ are some (non-weighted) Haar functions. 
\end{df}

\begin{df}
 Let $m, n\in \N$.  An elementary dyadic shift with parameters $m$, $n$ 
is an operator given by
\[
\sha f := \sum_{Q\in \cD}\sum_{\substack{ Q', Q''\in \cD,  Q', Q''\subset Q, \\ \ell(Q')=2^{-m}\ell(Q),\, \ell(Q'') =2^{-n}\ell(Q)}} |Q|^{-1}(f, h_{Q'}^{ Q''}) h_{Q''}^{Q'}
\]
where $h_{Q'}^{ Q''}$ and $h_{Q''}^{Q'}$ are (non-weighted) Haar functions for the cubes $Q'$ and $Q''$ respectively, 
subject to normalization
\begin{equation}
\label{norm1}
\|h_{Q'}^{ Q''}\|_\infty\cdot \|h_{Q''}^{Q'}\|_\infty \le 1.
\end{equation}
Notice that this implies, in particular, that
\begin{equation}
\label{sha1}
\sha f(x)=\sum_{Q\in \cD} |Q|^{-1} \int_Q a\ci Q(x,y)f(y) dy\,,\qquad \supp a\ci Q\subset Q\times Q, \ \|a\ci Q\|_{\infty}\le 1\,,
\end{equation}
where 
\begin{equation}
\label{sha-aQ}
a\ci Q (x,y) = \sum_{\substack{ Q', Q''\in \cD,  Q', Q''\subset Q, \\ \ell(Q')=2^{-m}\ell(Q),\, \ell(Q'')
=2^{-n}\ell(Q)}} h_{Q''}^{Q'}(x) h_{Q'}^{Q''}(y). 
\end{equation}
The number $\max(m,n)$ is called the \emph{compexity} of the dyadic shift. 
\end{df}

\begin{df}
If in the above definition we allow some (or all) $h\ci{Q'}$, $h\ci{Q''}$ to be \emph{generalized} Haar functions, we get what we will call an \emph{elementary generalized dyadic shift}.

A  dyadic shift with parameters $m$ and $n$ is a sum of at most $(2^d)^2$ elementary dyadic shifts (with parameters $m$ and $n$). If we allow some (or all) of the elementary dyadic shifts to be \emph{generalized} ones, we get the \emph{generalized} dyadic shift. 
\end{df}

\begin{rem*}The paraproduct $\Pi$ is an elementary generalized dyadic shift with parameters $0$, $1$, provided that $ \|h\ci Q\|_\infty \le 1$ for all cubes $Q$. 
\end{rem*}

\begin{rem*}
The main difference between dyadic shifts and generalized ones is that a dyadic shift is always a bounded operator in $L^2$ (assuming the normalization \eqref{norm1}), while for the boundedness of a generalized dyadic shift some additional conditions are required. 
\end{rem*}

We always think that our dyadic shifts $\sha$ are {\it finite} dyadic shifts meaning that only finitely many $Q$'s are involved in its definition above. All estimates will be independent of this finite number.

In the present section we consider a two weight $T(1)$ theorem for dyadic shifts. We fix two measures $\mu$, $\nu$ on $\mathbb{R}^d$.  Finite dyadic shifts are integral operators with kernel
$$
A(x,y) =\sum_{Q\in \cD} a\ci Q(x,y),
$$
the sum being well defined as it is finite.
We define now
$$
\sha_{\mu} f (x):=\int A(x,y)f(y)\,d\mu(y),
$$ 
and its adjoint $\sha^*_\nu$
\[
\sha_\nu^* g (y) = \int \overline A(x, y) g(x) d\nu(x).
\]

We need the notation
$$
[\mu,\nu]_{A_2}:= \sup_I\langle \mu\rangle_I\langle \nu\rangle_I\,,
$$
where $\langle \sigma\rangle_I:= |I|^{-1}\sigma(I)$.

The following theorem is the first new main result of this paper. It is essentially a quantified version of Theorem 2.3 of \cite{NTV6}.

%
%
%
%
%
%

\begin{thm}
\label{sh2w}
Let $\sha$ be an elementary generalized dyadic shift with parameters $m$ and $n$. Let us suppose that there exists a constant $B$ such that for any $Q\in\cD$ we have
\begin{equation}
\label{T12wsh}
\int_Q |\sha_{\mu}\1_{Q}|^2 d\nu \le B\mu(Q)\,,\,\qquad  \int_Q |\sha_\nu^* \1_{Q}|^2 d\mu \le B \nu(Q)\,.
\end{equation}
Then  
\begin{equation}
\label{T1sh}
\|\sha_{\mu}f\|_{\nu} \le C \left( 2^{d/2} (r+1) \left(  B^{1/2}+ [\mu,\nu]_{A_2}^{1/2} \right) + r^2 [\mu, \nu]\ci{A_2}^{1/2} \right)\|f\|_{\mu}\,.
\end{equation}
where $r=\max(m,n)$, and $C$ is an absolute constant. 
\end{thm}

The idea of the proof of this theorem  is quite simple. The operator $\sha^\mu$ is represented essentially as the sum of \emph{weighted paraproducts}, which are estimated using condition \eqref{T12wsh}  and the operator with finitely many diagonals, which is estimated by $C[\mu, \nu]_{A_2}^{1/2}$.

Take two test functions $f,g$. Using martingale difference decomposition \eqref{mdd-mu} we can decompose
\[
f = \sum_{Q\in\cD: \ell(I) =2^k} \E^\mu\ci Q f + \sum_{Q\in\cD: \ell(I) <2^k} \Delta^\mu\ci Q f, \qquad 
g = \sum_{Q\in\cD: \ell(I) =2^k} \E^\nu\ci Q g + \sum_{Q\in\cD: \ell(I) <2^k} \Delta^\nu\ci Q g .
\]

We want to estimate the bilinear form $\La \sha^\mu f, g \Ra_\nu$. We will first concentrate on the nontrivial case $f=\sum_{Q\in\cD} \Delta^\mu\ci Q f$, $g=\sum_{Q\in\cD} \Delta^\nu\ci Q g$; adding the terms $\sum_{Q\in\cD: \ell(I) =2^k} \E^\mu\ci Q f$ and $\sum_{Q\in\cD: \ell(I) =2^k} \E^\nu\ci Q g$ will be easy.%
\footnote{In fact, we will only apply this theorem in the situation when a martingale difference decompositions not involving $\E\ci Q^\mu$ and $\E\ci Q^\nu$ are possible.}

\subsection{Weighted paraproducts.}  Fix an integer $r$. Then   the paraproduct $\Pi^{\mu}=\Pi^{\mu}_{\sha}$, acting (formally) from $L^2(\mu)$ to $L^2(\nu)$ is defined as
\[
\Pi^{\mu} :=\sum_{Q\in \cD} \E^{\mu}_Q f \sum_{\substack{R\in \cD, R\subset Q, \\ \ell(R) =2^{-r}\ell(Q)}}  \Delta_R^{\nu}\sha_{\mu}  1_{Q}\,.
\]
The paraproduct $\Pi^{\nu}=\Pi^{\nu}_{\sha^*}$, acting (formally) from $L^2(\nu)$ to $L^2(\mu)$, is defined similarly
\[
\Pi^{\nu} :=\sum_{Q\in \cD} \E^{\nu}_Q f\sum_{\substack{R\in \cD, R\subset Q,\\ \ell(R) =2^{-r}\ell(Q)}} \Delta_R^{\mu}\sha_{\nu}^* 1_{Q}\,.
\]

Notice that if $r\ge n$, then for any $f\in L^1\ti{loc}(\mu)$ such that $f\bigm|_{Q} \equiv 1$, and for any $R\in\cD$ such that $ R\subset Q$ and $\ell(R)\le 2^{-r}\ell(Q)$, we have
\begin{equation}
\label{upping1}
\Delta_R^{\nu}\sha_{\mu}  f= \Delta_R^{\nu}\sha_{\mu}  \1_{Q}.
\end{equation}
Indeed, in the decomposition 
\[
\La \sha_\mu (\1\ci Q - f ), h^\nu\ci R \Ra_\nu = \sum_{I\in\cD} \sum_{\substack{I', I''\in\cD, \, I', I''\subset I \\ \ell(I') =2^{-m}\ell(I), \,\ell(I'')=2^{-n}\ell(I)} } \La \1\ci Q - f , h\ci{I'}\Ra_\mu
\La h\ci{I''}, h^\nu\ci R \Ra_\nu 
\]
only the terms with $I'\not\subset Q$ and $I''\subset R$ can give a non-zero contribution. But the inclusions $I''\subset R\subset Q$ together with size conditions  on $I''$ and $R$ imply that 
\[
\ell(I) = 2^{n}\ell(I'')\le 2^r \ell(I'')\le 2^r\ell(R) \le \ell(Q), 
\]
 so $I\subset Q$ (because $I\cap Q \supset I''\ne\varnothing$, so the inclusion of the dyadic cubes is determined by their sizes). But the inclusion $I\subset Q$ implies $I'\subset Q$, so the conditions $I'\not\subset Q$ and $I''\subset R$ are incompatible. 

The  equality \eqref{upping1} means that for $r\ge n$ we can  replace $\1\ci Q$ by $1$, bringing our definition of the paraproduct more in line with the classical one. 

\begin{lm}
\label{l2.2}
Let $Q,R\in\cD$, and let $r\ge n$. Then for the paraproduct $\Pi^\mu=\Pi^\mu_{\sha^*}$ 
defined above
\begin{enumerate}
    \item If $\ell(R) \ge  2^{-r}\ell(Q)$ then $\La \Pi^\mu h^\mu_Q, 
h^\nu_R\Ra_\nu = 0$ for all weighted Haar functions $h^\mu_Q$ and 
$h^\nu_R$.

\item If $R\not\subset Q$, then 
$\La \Pi^\mu h^\mu_Q, h^\nu_R\Ra_\nu = 0$ for all weighted Haar 
functions $h^\mu_Q$ and $h^\nu_R$.

\item If $\ell(R) <2^{-r} \ell(Q)$, then for all weighted Haar functions 
$h^\mu_Q$ and $h^\nu_R$
$$
\La \Pi^\mu h_Q^\mu, h_R^\nu\Ra_\nu = \La \sha_\mu h_Q^\mu, h_R^\nu\Ra_\nu;
$$
in particular, if $R\not\subset Q$, then both sides of the equality are 
$0$.

\end{enumerate}
\end{lm}

\begin{proof}
Let us use $Q'$ and $R'$ for the summation indices in the paraproduct, 
i.e.~let us write
\[
\Pi^\mu h_Q^\mu := \sum_{Q'\in \cD} \E_{Q'}^\mu h_Q^\mu 
\sum_{\substack{R'\in \cD,\ R'\subset Q', \\ 
\ell(R')=2^{-r}\ell(Q')}} \Delta_{R'}^\nu \sha_\mu \1\ci{Q'}.
\]
Since $h^\nu_R$ is orthogonal to ranges of all projections 
$\Delta_{R'}^\nu$ except $\Delta_R^\nu$ we can write
\begin{equation}
\label{Pi-hq-hr}
\La \Pi^\mu h^\mu_Q, h_R^\nu\Ra_\nu = \La (E_{Q'}^\mu h^\mu_Q) \Delta_R^\nu 
\sha_\mu\1\ci{Q'}, h_R^\nu \Ra_\nu = a\La \sha_\mu\1\ci{Q'}, h_R^\nu \Ra_\nu 
\end{equation}
where $Q'$ is the ancestor of $R$ of order $r$ (i.e.~the cube $Q'\supset R$ such that $\ell(Q')= 2^r \ell(R)$) and $a$ is the value of $E\ci{Q'}^\mu h^\mu\ci{Q}$ on $Q'$, $E^\mu\ci{Q'} h^\mu\ci{Q} = a \1\ci{Q'}$.

It is easy to see that $E^\mu_{Q'} h^\mu_Q \not\equiv 0$ (equivalently $a\ne 0$) only if $Q'\subsetneqq 
Q$.  
Therefore, see \eqref{Pi-hq-hr},  
\[
\La \Pi^\mu h^\mu_Q, h_R^\nu\Ra_\nu\ne 0
\] 
only if $Q'\subsetneqq Q$
 and statements \cond1 and \cond2  of the lemma follow immediately. 

Indeed, if $\ell(R)\ge 2^{-r}\ell(Q)$ and $\ell(Q') =2^r \ell(R)$, the inclusion $Q'\subsetneqq 
Q$ is impossible, so  
\[
 \La \Pi^\mu h^\mu_Q, h_R^\nu\Ra_\nu = 0, 
\]
and the statement \cond1 is proved. 

If  $R\not\subset Q$, then the inclusion $Q'\subsetneqq 
Q$ (which, as it was discussed above, is necessary for \linebreak $\La \Pi^\mu h^\mu_Q, h_R^\nu\Ra_\nu\ne0$)
implies that $R\not\subset Q'$. This means that $Q'$ is not an ancestor of $R$, however \eqref{Pi-hq-hr} again shows that for $Q'$ to be an ancestor of $R$  is necessary for  $\La \Pi^\mu h^\mu_Q, h_R^\nu\Ra_\nu\ne0$.

Let us prove statement \cond3. Let $\ell(R)<2^{-r}\ell(Q)$. If $R\not\subset 
Q$ then by the statement \cond2 of the lemma $\La \Pi^\mu h^\mu_Q, h^\nu_R\Ra_\nu 
=0$. On the other hand if $M$ is the ancestor of order $r$ of $R$, then $Q\cap M=\varnothing$, thus by \eqref{upping1}
\[
\La \sha_\mu h^\mu_Q, h^\nu_R\Ra_\nu = \La \sha_\mu 0\cdot\1\ci M, h^\nu_R\Ra_\nu=0. 
\] 
 So, we only 
need to consider the case $R\subset Q$.

Let $Q_1$ be the ``child'' of $Q$ containing $R$ (i.e.~$R\subset 
Q_1\subset Q$, $\ell(Q_1) =\ell(Q)/2$), and let $b$ be the value of 
$h^\mu_Q$ on $Q_1$. Then, since $\ell(R) \le 2^{-r} \ell(Q_1)$, \eqref{upping1}  implies  that
\[
\La \sha_\mu h_{Q}^\mu , h^\nu_R\Ra_\nu = b \La \sha_\mu\1\ci{Q_1} , h^\nu_R\Ra_\nu
\]
On the other hand we have shown before, see \eqref{Pi-hq-hr} that
\[
\La \Pi^\mu h^\mu_Q, h_R^\nu\Ra_\nu = \La (E_{Q'}^\mu h^\mu_Q) \Delta_R^\nu 
\sha_{\mu}\1\ci{Q'}, h_R^\nu \Ra_\nu
\]
where $Q'\in\cD$ is the ancestor of order $r$ of $R$, meaning that $R\subset Q' 
$, $\ell(Q')=2^{r}\ell(R)$. Therefore 
$Q'\subset Q_1$ and so $E_{Q'}^\mu h_Q^\mu =b\1\ci{Q'}$. We also know, see \eqref{upping1}, that because $Q'\subset Q_1$ we have equality $\Delta_R^\nu 
\sha_\mu \1\ci{Q'} = \Delta_R^\nu \sha_\mu\1\ci{Q_1}$. Thus we can continue:
\[
\La \Pi^\mu h^\mu_Q, h_R^\nu\Ra_\nu  =
b\La  \Delta_R^\nu \sha_\mu \1\ci{Q'}, h_R^\nu \Ra_\nu = b\La  \Delta_R^\nu 
\sha_\mu\1\ci{Q_1}, h_R^\nu \Ra_\nu= b\La  \sha_\mu\1\ci{Q_1}, h_R^\nu \Ra_\nu.
\]
Therefore $\La \Pi^\mu h^\mu_Q, h_R^\nu\Ra_\nu = \La \sha_{\mu} h_{Q}^\mu , 
h^\nu_R\Ra_\nu$, and the lemma is proved.
\end{proof}

\subsection{Boundedness of the weighted  paraproduct}

We will need the following well known theorem.

Let $f\ci{\!R}:= \frac1{\mu(R)} \int_R f\,d\mu$ be the average of the 
function $f$ with respect to the measure $\mu$.

\begin{thm}[Dyadic Carleson Embedding Theorem]
\label{t:Carl}
If the numbers $a\ci Q\ge 0$, $Q\in \cD$, satisfy the following Carleson 
measure condition
\begin{equation}
\label{C-cond}
\sum_{Q\subset R} a\ci Q \le \mu(R),
\end{equation}
then for any $f\in L^2(\mu)$
\[
\sum_{R\in \cD} a\ci R |f\ci{\!R}|^2 \le 4\cdot \|f\|_{L^2(\mu)}^2.
\]
\end{thm}

This theorem is very well known, cf \cite{Duren1970}. Usual proofs are based on a stopping time argument and the dyadic maximal inequality; the constant $4$ appears as $2^2$, where $2$ is the norm of the dyadic maximal operator on $L^2(\mu)$. For an alternative proof using the Bellman function method, see \cite{NTV1}.
It was also proved in \cite{NTV-name} 
that the constant $4$ is optimal. We should mention that in 
\cite{NT}, \cite{NTV-name} this theorem was proved for $\R^1$, but the 
same proof works for general martingale setup. A proof for $\R^2$ was 
presented in \cite{NTV1}, and the same proof works for $\R^d$.

Let us now show that the paraproduct $\Pi=\Pi^\mu_{\sha}$ is bounded. Ranges 
of the projections $\Delta^\nu_R$ are mutually orthogonal, so to prove 
the boundedness of the paraproduct $\Pi^\mu_{\sha}$ it is sufficient to show 
that the numbers
\[
a\ci Q := \sum_{\substack{R\in \cD, R\subset Q\\ \ell(R) = 
2^{-r}\ell(Q)}} \|\Delta^\nu_R \sha_\mu \1\ci R \|^2_{L^2(\nu)} 
\]
satisfy the Carleson Measure Condition  \eqref{C-cond} from Theorem 
\ref{t:Carl}. Let us prove this.

Consider a cube $\wt Q$. We want to show that
$$
\sum_{Q\subset \wt Q} \sum_{\substack{R\in \cD, R\subset Q\\ \ell(R) = 
2^{-r}\ell(Q)}} \|\Delta^\nu_R \sha_\mu\1\ci Q \|_{L^2(\nu)}^2 \le 
B \mu(\wt Q).
$$
By  \eqref{upping1} we can replace $\1\ci Q$ by $\1\ci{\wt Q}$, so 
the desired estimates becomes
\[
\sum_{\substack{R\in \cD, R\subset \wt Q\\ \ell(R) \le 2^{-r}\ell(\wt 
Q)}} \|\Delta^\nu_R \sha_\mu\1\ci{\wt Q} \|_{L^2(\nu)}^2 \le 
\sum_{R\subset \wt Q} \|\Delta^\nu_R \sha_\mu\1\ci{\wt Q} \|_{L^2(\nu)}^2
\le \|\1\ci{\wt Q} \sha_\mu\1\ci{\wt Q} \|_{L^2(\nu)}^2 .
\]
By the assumption  of Theorem \ref{sh2w}, see \eqref{T12wsh},
\[
\|\1\ci{\wt Q} \sha_\mu\1\ci{\wt Q} \|_{L^2(\nu)}^2 :=\int_{\wt Q} | 
\sha_\mu\1\ci{\wt Q} |^2d\nu \le B\mu(\wt Q)
\]
and so the sequence $a\ci Q$, $Q\in \cD$ satisfies the condition 
\eqref{C-cond}. Thus the norm of the paraproduct $\Pi^\mu$ is bounded by $C B^{1/2}$ (we can pick $C=2$ here) and similarly for $\Pi^\nu$.  \hfill\qed

\subsection{Boundedness of  \texorpdfstring{$\sha$}{S}: essential part}

Let $f\in L^2(\mu)$, $g\in L^2(\nu)$, $\|f\|_\mu, \|g\|_\nu\le 1$.
We want to estimate
$|\La \sha_\mu f, g\Ra_\nu|$.

Consider first $f$ and $g$ of form 
\[
f= \sum_{Q\in\cD} \Delta^\mu_Q f, \qquad g= \sum_{R\in\cD} \Delta_R^\nu g, \qquad \|f\|_\mu\le 1, \ \|g\|_\nu\le 1. 
\]
Then by Lemma \ref{l2.2}
\begin{align}
\label{sha-f-g}
\Bigl\La \sha_\mu  f,  g 
\Bigr\Ra_\nu =
\La \Pi_{\sha_\mu}^\mu f, g \Ra_\nu + \La f, \Pi_{\sha_\nu^*}^\nu g \Ra_\nu  + 
 \sum_{\substack{Q,R\in \cD,  \\ 2^{-r} \le \ell(R)/\ell(Q) \le 
2^r }}  \La \sha_\mu \Delta_Q^\mu f, \Delta_R^\nu g\Ra_\nu 
\end{align}
We know that the paraproducts $\Pi_{\sha_{\mu}}^\mu$ and $\Pi_{\sha_{\nu}^*}^\nu$ are 
bounded, so the first two terms  can be estimated together by $4B^{1/2}$. Thus it remains to estimate the last sum. 

It is enough to estimate the operator $S$
\[
\La S f, g\Ra_\nu : = \sum_{\substack{Q,R\in \cD \\ 2^{-r}\ell(Q) \le \ell(R)\le \ell(Q) }}  \La \sha_\mu \Delta_Q^\mu f, \Delta_R^\nu g \Ra_\nu 
\]
because the sum over $2^{-r} \ell(R) \le \ell(Q) <\ell(R)$ is estimated similarly. The operator  $S$ can be split as $S=\sum_{k=0}^r S_k$, where the 
\[
\La S_k f, g\Ra_\nu : = \sum_{\substack{Q,R\in \cD \\   \ell(R)=2^{-k}\ell(Q) }}  \La \sha_\mu \Delta_Q^\mu f, \Delta_R^\nu g\Ra_\nu 
\] 

Each $S_k$ can be in turn decomposed as $S_k =\sum_{j\in\Z} S_{k,j}$, where 
\[
\La S_{k,j} f, g\Ra_\nu : = \sum_{\substack{Q,R\in \cD \\ \ell(Q)=2^{j} \\ \ell(R)=2^{j-k} }}  \La \sha_\mu \Delta_Q^\mu f, \Delta_R^\nu g\Ra_\nu 
\]
For a fixed $k$ the ranges $\ran S_{k, j}$, $j\in\Z$ are mutually orthogonal in $L^2(\nu)$, and the dual ranges $\ran S^*_{k, j}$, $j\in\Z$ are mutually orthogonal in $L^2(\mu)$. Therefore $\|S_k\|_\le \max_{j\in\Z} \|S_{k,j}\|$, so we only need to uniformly estimate individual operators $S_{k,j}$. 

So, if 
\[
f_j =\sum_{Q\in\cD: \ell(Q) =2^j} \Delta^\mu\ci Q f, \qquad g_{j-k}=\sum_{R\in\cD: \ell(R) =2^{j-k}} \Delta^\nu\ci Q g
\]
it is sufficient to estimate $\La S_{k,j} f_j, g_{j-k}\Ra_\nu = \La \sha_\mu f_j, g_{j-k}\Ra_\nu$. 

We can decompose the  operator $S_{k,j} $ into \emph{interior} and \emph{outer} parts
\begin{align*}
\La S_{k,j} f, g\Ra_\nu 
&= \sum_{\substack{Q, R\in\cD:R\subset Q \\ \ell(Q) =2^j, \ell(R)=2^{j-k}}} \La \sha_\mu \Delta^\mu\ci Q f , \Delta^\nu\ci R g\Ra_\nu +
\sum_{\substack{Q, R\in\cD:R\cap Q =\varnothing \\ \ell(Q) =2^j, \ell(R)=2^{j-k}}} \La \sha_\mu \Delta^\mu\ci Q f , \Delta^\nu\ci R g\Ra_\nu 
\\
& =: \La S_{k,j}^{\text{int}} f, g\Ra_\nu + \La S_{k,j}^{\text{out}} f, g\Ra_\nu 
\end{align*}

Let us estimate $S_{k,j}^{\text{out}}$. For cubes $Q, R\in\cD$, $R\cap Q =\varnothing$, $\ell(Q) =2^{j}$, $\ell(R) =2^{j-k}$ and the corresponding weighted Haar functions $h^\mu\ci Q$ and $h^\nu\ci R$ we can write 
\begin{equation}
\label{S_out_hQ_hR}
\La S_{k, j}^{\text{out}} h^\mu\ci Q, h^\nu\ci R \Ra_\nu  = \La \sha_\mu h^\mu\ci Q, h^\nu\ci R\Ra_\nu
= \sum_{M\in\cD} |M|^{-1}\int_{M\times M} a_M(x,y) h^\mu\ci Q (y) \overline{h^\nu\ci R }(x) d\mu(y) d\nu(x)  
\end{equation}
where the kernels $a_M$ are from \eqref{sha1}. 

If $\ell(M)\le \ell(Q) =2^j$, then the cube $M$ cannot contain both $Q$ and $R$ (because $R\cap Q =\varnothing$), so the corresponding integral in \eqref{S_out_hQ_hR} is $0$. On the other hand, if $\ell(M)> 2^r \ell(Q) $, $r=\max(m, n)$ being the complexity of the dyadic shift $\sha$, then for any $x$ the function $a\ci M(x, \fdot) $ is constant on $Q$, so the corresponding integral in  \eqref{S_out_hQ_hR} is again $0$. 

So in \eqref{S_out_hQ_hR} we only need to count $M$, $2^j < \ell(M)\le 2^{j+r}$, and therefore we can write
\begin{align*}
| \La S_{k, j}^{\text{out}} h^\mu\ci Q, h^\nu\ci R \Ra_\nu | 
& = \biggl|\sum_{s=j+1}^{j+r} \int_{\R^d\times\R^d} A_s (x, y) h\ci Q^\mu(y) \overline{h\ci R^\nu}(x) d\mu(y) d\nu(x) \biggr| 
\\
& \le \sum_{s=j+1}^{j+r} \int_{\R^d\times\R^d} | A_s (x, y)| \cdot | h\ci Q^\mu(y) |\cdot | h\ci R^\nu(x) | d\mu(y) d\nu(x), 
\end{align*}
where $A_s (x, y)  := \sum_{M\in\cD: \ell(M) =2^s} |M|^{-1} a\ci M(x, y) $. 

Adding extra non-negative terms (with $R\subset Q$) we can estimate
\begin{align*}
| \La S_{k,j}^{\text{out}} f, g\Ra_\nu | 
& \le \sum_{s=j+1}^{j+r} \sum_{\substack{Q, R\in\cD:R\cap Q =\varnothing \\ \ell(Q) 
=2^j, \ell(R)=2^{j-k}}} \int_{\R^d\times \R^d} | A_s (x, y)| \cdot | \Delta\ci Q^\mu f(y) | \cdot |\Delta\ci R^\nu g(x) | d\mu(y) d\nu(x) 
\\
& \le \sum_{s=j+1}^{j+r} \int_{\R^d\times \R^d} | A_s (x, y)| \cdot |f_j(y)| \cdot | g_{j-k}(x) | d\mu(y) d\nu(x) 
\end{align*}

But each integral operator with kernel $|A_s|$ is the direct sum of the operators with kernels $|M|^{-1} |a\ci M|$, $M\in\cD$, $\ell(M) =2^s$ (recall that $a\ci M$ is supported on  $ M\times M$). 

Since $\|a_M\|_\infty\le 1$ we can estimate the Hilbert--Scmidt norm
\[
\int_{M\times M} |M|^{-2} |a\ci M(x, y)|^2 d\mu(y) d\nu(x) \le [\mu, \nu]_{A_2}.  
\]
so the norm each operator with kernel $|M|^{-1} |a\ci M(x, y)|$ is at most $[\mu, \nu]_{A_2}^{1/2}$. 
Therefore the norm of each operator with kernel $|A_s(x, y)|$ is estimated by $[\mu, \nu]_{A_2}^{1/2}$, and summing in  $s$
we get 
\begin{equation}
\label{Sjk-out}
\| S_{k,j}^{\text{out}} \|\ci{L^2(\mu)\to L^2(\nu)} \le r [\mu, \nu]_{A_2}^{1/2}
\end{equation}

To estimate the norm of $S_{k,j}^{\text{int}}$ we need the following simple lemma 

\begin{lm}
\label{l:1Q_sha_hQ}
In the assumptions of Theorem \ref{sh2w}
\[
\| \1\ci Q \sha_\mu h\ci Q^\mu \|_\nu^2 \le 2^d (B + 4 [\mu,\nu]_{A_2}) \|h^\mu\ci Q\|_\mu^2. 
\]
for any $\mu$-Haar function $h\ci Q^\mu$. 
\end{lm}
\begin{proof}
Let $Q_k$, $k=1, 2, \ldots, 2^d$ be the dyadic children of $Q$. A $\mu$-Haar function $h^\mu\ci Q$ can be represented as 
\begin{equation}
\label{hQmu}
h^\mu\ci Q =\sum_{k=1}^{2^d} \alpha_k \1\ci{Q_k} , \qquad \sum_{k=1}^{2^d} \alpha_k \mu(Q_k) =0. 
\end{equation}
and 
\begin{equation}
\label{norm_hQmu}
\| h^\mu\ci Q \|_\mu^2 = \sum_{k=1}^{2^d} |  \alpha_k |^2 \mu(Q_k). 
\end{equation}

By assumption \eqref{T12wsh} of Theorem \ref{sh2w}
\begin{equation}
\label{1sha1}
\| \1\ci{Q_k} \sha_\mu \1\ci{Q_k} \|_\nu^2 \le B \mu(Q_k) .
\end{equation}

Let us estimate $\| \1\ci{Q\setminus Q_k} \sha_\mu \1\ci{Q_k} \|_\nu$. We know that 
\[
\sha_\mu \1\ci{Q_k}(x) = \sum_{M\in\cD} |M|^{-1} \int_{Q_k} a\ci M(x, y) \1\ci{Q_k}(y) d\mu(y). 
\]
Since the functions $a\ci M$ are supported on $M\times M$,  only the terms with $M\supset Q$ can give a non-zero contribution for $x\notin Q_k$. Therefore, summing the geometric series we get that 
\[
| \sha_\mu \1\ci{Q_k}(x)| \le 2 \mu(Q_k) |Q|^{-1}\qquad \forall x\notin Q_k. 
\] 
Then 
\[
\| \1\ci{Q\setminus Q_k} \sha_\mu \1\ci{Q_k} \|_\nu^2 \le 4 \mu(Q_k)^2 |Q|^{-2} \nu(Q) , 
\]
and combining this estimate with \eqref{1sha1} we get
\begin{align*}
\| \1\ci{Q} \sha_\mu \1\ci{Q_k} \|_\nu^2 & \le B \mu(Q_k) + 4 \mu(Q_k)^2 |Q|^{-2} \nu(Q) \\
& \le B \mu(Q_k) + 4 \mu(Q_k) \mu(Q) |Q|^{-2} \nu(Q)  \\
& \le (B + 4 [\mu, \nu]_{A_2} ) \mu(Q_k) 
\end{align*}
Therefore, we can get recalling \eqref{hQmu} and \eqref{norm_hQmu}
\begin{align*}
\| \1\ci Q \sha_\mu h\ci Q^\mu \|_\nu 
&\le \sum_{k=1}^{2^d} |\alpha_k| \| \1\ci{Q} \sha_\mu \1\ci{Q_k} \|_\nu \\
& \le  \left(B + 4 [\mu, \nu]_{A_2} \right)^{1/2} \sum_{k=1}^{2^d} |\alpha_k| \mu(Q_k)^{1/2} \\
& \le \left(B + 4 [\mu, \nu]_{A_2} \right)^{1/2} 2^{d/2} \left( \sum_{k=1}^{2^d} |\alpha_k|^2 \mu(Q_k) \right)^{1/2} \\
& = 2^{d/2} \left(B + 4 [\mu, \nu]_{A_2} \right)^{1/2} \| h^\mu\ci Q\|_\mu 
\end{align*}
\end{proof}

Using the above Lemma \ref{l:1Q_sha_hQ}, we can easily estimate $S_{k,j}^{\text{int}}$. Namely, 
\begin{align*}
\| S_{k,j}^{\text{int}} f_j \|_\nu^2 & = \sum_{Q\in\cD: \ell(Q) =2^j} 
\Biggr\| \sum_{R\subset Q: \ell(R) =2^{j-k}} \Delta^\nu\ci R\sha_\mu \Delta^\mu\ci Q f \Biggl\|_\nu^2 \\
& \le \sum_{Q\in\cD: \ell(Q) =2^j} \left\| \1_Q\sha_\mu \Delta^\mu\ci Q f \right\|_\nu^2 \\
&\le  2^d (B + 4 [\mu,\nu]_{A_2}) \sum_{Q\in\cD: \ell(Q) =2^j} \left\| \Delta^\mu\ci Q f \right\|_\mu^2\\
& = 2^d (B + 4 [\mu,\nu]_{A_2}) \|f_j\|_\mu^2.
\end{align*}
Combining this with the estimate \eqref{Sjk-out} of $\| S_{k,j}^{\text{out}} \|$, we get that
\[
\left\| S_{k,j} \right\|\ci{L^2(\mu)\to L^2(\nu)} \le 2^{d/2} (B + 4 [\mu,\nu]_{A_2})^{1/2} + r  [\mu,\nu]_{A_2}^{1/2}) . 
\]
Since the operator $S_k$ is the orthogonal sum of $S_{k,j}$, we get the same estimate for $\|S_k\|$. To get the estimate for $\|S\|$, $S =\sum_{k=0}^r S_k $, we just multiply the above estimate by $r+1$. 

Adding in \eqref{sha-f-g} all the estimates together we get that for $f$ and $g$ of form 
\[
f= \sum_{Q\in\cD} \Delta^\mu_Q f, \qquad g= \sum_{R\in\cD} \Delta_R^\nu g, \qquad \|f\|_\mu\le 1, \ \|g\|_\nu\le 1,  
\]
we have 
\begin{equation}
\label{sha_f_f-est}
|\La \sha_\mu f, g \Ra_\nu | \le 4 B^{1/2} + 2\cdot (r+1) [ 2^{d/2}(B + 4 [\mu,\nu]_{A_2})^{1/2} + r  [\mu,\nu]_{A_2}^{1/2}];  
\end{equation}
the first term here comes from the paraproducts, and the extra factor $2$ in second term is to take into account the sum over $\ell(Q)<\ell(R)$ in \eqref{sha-f-g}.

\subsection{Boundedness of  \texorpdfstring{$\sha$}{dyadic shift}: some little details}

We are almost done with the proof of Theorem \ref{sh2w}, modulo a little detail: for arbitrary measures $\mu$ functions $f\in L^2(\mu)$ do not admit martingale difference decomposition $f=\sum_{Q\in\cD} \Delta^\mu\ci Q f$.

 Each compact subset of $\R^d$ is contained in at 
most $2^d$ cubes of the same size as the size of this compact subset, so let $Q_k$, $k=1, 2, \ldots, 2^d$ 
be the dyadic cubes of some size $2^N$ containing supports of $f$ and $g$. The correct decomposition is given by \eqref{mdd-mu} which reads as 
\begin{equation}
\label{mdd-f}
f = \sum_{Q\in\cD: \ell(Q) =2^k} \E^\mu\ci Q f + \sum_{Q\in\cD: \ell(Q) \leq 2^k} \Delta^\mu\ci Q f
\end{equation}
(here $k$ is an arbitrary but fixed integer), and similarly for $g\in L^2(\nu)$. 
\begin{equation}
\label{mdd-g}
g = \sum_{Q\in\cD: \ell(Q) =2^k} \E^\nu\ci Q g + \sum_{Q\in\cD: \ell(Q) \leq 2^k} \Delta^\nu\ci Q g. 
\end{equation}
so  we need to estimate some extra terms. Of course, in the situation when we apply the theorem ($d\mu =w dx$, $d\nu =w^{-1} dx$, $w$ satisfies the $A_2$ condition) $f$ and $g$ can be represented via martingale difference decomposition, although some explanation will still be needed. 

Fortunately, there is a very simple way to estimate the extra terms. Let us say that  dyadic cubes $Q, R\in \cD$ are \emph{relatives} if they have a common ancestor, i.e.~ a cube $M\in\cD$ such that $Q, R\subset M$. The importance of the notion of relatives stems from the trivial observation that if the cubes $Q$ and $R$ are not relatives, then $\sha_\mu \1\ci Q \equiv 0$ on $R$. 

It is sufficient to prove the estimate on a dense set of compactly 
supported functions. For compactly supported functions $f$ and $g$ only finitely many terms $\E^\mu\ci Q f$ and $ \E^\nu\ci Q g$ in the decompositions \eqref{mdd-f} and \eqref{mdd-g} are non-zero. Let us slit the collection of corresponding cubes into equivalence classes of relatives, and for each equivalence class find a common ancestor (it is always possible because of finiteness). 

Denote by  $\ccA$  the set of these common ancestors. Then we can write instead of \eqref{mdd-f} and \eqref{mdd-g}
\begin{align}
\label{mdd-f1}
f & = \sum_{Q\in\ccA} \E^\mu\ci Q f +  \sum_{Q\in\ccA} \sum_{R\in \cD: R\subset Q} \Delta^\mu\ci R f =: f\ti e + f\ti d,
\\
\label{mdd-g1}
g & = \sum_{Q\in\ccA} \E^\nu\ci Q g +  \sum_{Q\in\ccA} \sum_{R\in \cD: R\subset Q} \Delta^\nu\ci R g = : g\ti e + g\ti d;
\end{align}
the indices ``e'' and ``d'' here mean \emph{expectation} and \emph{difference}. Let us decompose
\begin{align*}
\La \sha_\mu f , g\Ra_\nu & = \La \sha_\mu (f\ti e  + f\ti d ) , g\ti e + g\ti d \Ra_\nu \\
& = \La \sha_\mu f\ti e, g \Ra_\nu + \La \sha_\mu f\ti d, g\ti e  \Ra_\nu + \La \sha_\mu f\ti d, g\ti d \Ra_\nu
\end{align*}
The last term is estimated by \eqref{sha_f_f-est} (note that $\|f\|_\mu^2 = \|f\ti e\|_\mu^2 + \|f\ti d\|_\mu^2$ and similarly for $\|g\|_\nu^2$), so we just need to estimate the first two terms.  

Any two cubes $Q, Q'\in\ccA$, $Q\ne Q'$ are not relatives, so as we already mentioned $\sha_\mu \1\ci Q \equiv 0$ on any $Q'\in \ccA$, $Q'\ne Q$. Therefore
\begin{align*}
|\La \sha_\mu \E\ci Q^\mu f, g \Ra_\nu| = |\La \sha_\mu \E\ci Q^\mu f, g \1\ci Q \Ra_\nu| 
& \le \| \1\ci Q \sha_\mu \E\ci Q^\mu f \|_\nu  \| g \1\ci Q \|_\nu
\\
& \le B^{1/2} \| \E\ci Q^\mu f \|_\mu \| g \1\ci Q \|_\nu 
\end{align*}
(we use assumption \eqref{T12wsh} of theorem \ref{sh2w} for the last inequality). Summing over all $Q\in\ccA$ and applying Cauchy--Schwarz inequality we get 
\begin{align*}
|\La \sha_\mu  f\ti e , g \Ra_\nu| = \sum_{Q\in\ccA} |\La \sha_\mu \E\ci Q^\mu f, g \Ra_\nu| 
& \le B^{1/2}\sum_{Q\in\ccA} \| \E\ci Q^\mu f \|_\mu \| g \1\ci Q \|_\nu 
\\
& \le B^{1/2} \| f\ti e\|_\mu \|g\|_\nu
\le   B^{1/2} \| f\|_\mu \|g\|_\nu
\end{align*}

Similarly
\begin{align*}
|\La \sha_\mu  f\ti d , g\ti e  \Ra_\nu|=  |\La   f\ti d , \sha_\nu^* g\ti e  \Ra_\nu| \le B^{1/2} \| f\ti d\|_\mu \| g\ti e\|_\nu \le   B^{1/2} \| f \|_\mu \| g \|_\nu, 
\end{align*}
so in general case we just need to add $2 B^{1/2}$ to the right side of \eqref{sha_f_f-est}.

\section{Dyadic shifts and random lattices}
\label{DSRL}

In this section we use a probabilistic approach to decompose an arbitrary \cz operator as an average of simple blocks,
namely, the dyadic shifts investigated above. More precisely, we prove the following result, which is a variant of \cite{H}, Theorem~4.2. 
The decomposition  here is easier than in \cite{H}, and there is a reason for that: the shifts in \cite{H} needed to have an extra geometric property pertinent to being applied in conjunction with \cite{PTV1}. Here we do not need that as we are not basing our reasoning on a weighted $T1$ theorem of  \cite{PTV1}. The idea of such decomposition goes back to methods of non-homogeneous Harmonic Analysis exploited in \cite{NTV5}  or \cite{VolLip} for example.

\begin{thm}
\label{T=avg-shift}
Let $T$ be a \cz operator in $\R^d$ with parameter $\alpha$. Then $T$ can be represented as 
\[
T = C \int_\Omega \sum_{m, n\in\Z_+} 2^{- (m+n)\alpha /2}\,\, \sha^{\om}_{m,n}  \, d\bP(\om)
\]
where $\sha^{\om}_{m,n}$ is a dyadic shift with parameters $m, n$ in the lattice $\cD_\om$; the shifts with parameters $0, 1$ and $1, 0$ can be generalized shifts, and all other shifts are the regular ones. 

The constant $C$ depends only on the dimension $d$ and the parameters of the \cz operator $T$ (the norm $\|T\|\ci{L^2\to L^2}$, the smoothness $\alpha$, and the constant $C\ti{cz}$ in the \cz estimates). 
\end{thm}

%
%

\subsection{Getting rid of bad cubes}
Let $\cD_\om$, $\om\in\Omega$ be the translated dyadic lattice in $\R^d$ as defined in Section \ref{s:RDL}  and let $\bP$ be the canonical probability measure on $\Omega$ (also defined in Section \ref{s:RDL}).

Fix $r_0\in \N$. Let $\gamma=\frac{\alpha}{2(d+\alpha)}$, where $\alpha$ is the \cz parameter of the operator $T$. 
\begin{df*}  A cube $Q\in \cD_{\omega}$ is called \emph{bad}  if there exists a bigger cube $R\in \cD_{\omega}$ such that $\ell(Q)< 2^{-r_0} \ell(R)$ and 
\[
\dist (Q, R) < \ell(Q)^{\gamma}\ell(R)^{1-\gamma}. 
\]
\end{df*}

Let us introduce some probabilistic notation we will use in this section. 
Let $\bE=\bE_\Omega$ denote the expectation with respect to the probability measure $\bP$, 
\[
\bE_\Omega F = \bE_\Omega F(\om) =\int_\Omega F(\om)\, d\bP(\om);
\] 
slightly abusing the notation we will often write $ \bE_\Omega F(\om)$ to emphasize that $F$ is a random variable (depends on $\om$). 

For $k\in\Z$ let  $\sA_k$ be the sigma-algebra generated by the random variables $\om_j$, $j< k$, and let $\bE_{\sA_k}$ be the corresponding conditional expectation. 
Because of the product structure of $\Omega$, the conditional expectation $\bE_{\sA_k}$ is easier to understand: it is just the integration with respect to a part of variables $\om_j$. 

Namely, for $k\in \Z$ one can split $\om = (\tensor*[^k]\om{}, \om^k)$, where $\tensor*[^k]\om{} := (\om_j)_{j<k}$, $\om^k:= (\om_j)_{j\ge k}$, so $\Omega$ is represented as a product $\Omega =  \tensor*[^k]{\Omega}{} \times \Omega^k$. Note that the sets $\tensor*[^k]{\Omega}{}$ and $\Omega^k$ are probability spaces with respect to the standard product measures. We will use the same letter $\bP$ for these measures (probabilities), hoping that this will not lead to the confusion.

Denote by $\Omega^k[ \tensor*[^k]\om{}]$ the ``slice'' of $\Omega$, 
\[
\Omega^k [ \tensor*[^k]\om{} ] = \{ (\tensor*[^k]\om{}, \om^k): \om^k\in\Omega^k \}. 
\]
Then for almost all $\tensor*[^k]\om{}$, assuming that  $\om = (\tensor*[^k]\om{}, \om^k)$ we have
\begin{align*}
(\bE_{\sA_k} F)(\om) = \bE_{ \Omega^k[ \tensor*[^k]\om{}] } F := \int_{\Omega^k} F(\tensor*[^k]\om{}, \tilde\om^k) \, d\bP(\tilde\om^k),  
\end{align*}
so the conditional expectation $\bE_{\sA_k}$ is just the integration over slices. 

Finally, given a cube $Q\in \cD_\om$, $\ell(Q)=2^k$, denote by $\Omega[Q]$ the slice  $\Omega[Q] := \Omega^k[ \tensor*[^k]\om{}] $ for the particular choice of the parameters $\tensor*[^k]\om{} = (\om_j)_{j<k}$ determining the position of $Q$ (and of all cubes of size $2^k$).  The notation $\bE_{\Omega [Q]}$ then should be clear, and one also can define the conditional probability 
\begin{align*}
\bP\{\text{event} | Q\} := \bE_{\Omega [Q]} \1\ti{event} . 
\end{align*}

\begin{lm}
\label{pr}
$\pi\ti{bad} =\pi\ti{bad} (r_0, \gamma, d) :=\bP\{Q \,\textup{is bad}| Q\}\le C(d)\, 2^{-cr_0}$.
\end{lm}

In words: given a cube $Q$, the probability that it is bad is a constant depending only on $r_0$, $\gamma$ and $d$, and can be estimated as stated.

\begin{proof}
The proof is an easy exercise for the reader. 
\end{proof}

For now on let us fix a sufficiently large $r_0$ such that $\pi\ti{bad}<1$, so the probability of being good satisfies $\pi\ti{good} =1-\pi\ti{bad}>0$.

\begin{lm}
\label{l:avg1}
Let $T$ be a bounded operator in $L^2 =L^2(\R^d, dx)$. Then for all $f, g\in C^\infty_0$
\begin{align*}
\La T f, g\Ra 
=\pi^{-1}\ti{good} \int_\Omega \sum_{\substack{I, J \in \cD_\omega \\ \ell(I) \le \ell(J) \\ I \ \textup{is good} }}
\La T \Delta\ci I f , \Delta\ci J g \Ra \, d\bP(\omega) 
+
\pi^{-1}\ti{good}
\int_\Omega \sum_{\substack{I, J \in \cD_\omega \\ \ell(I) > \ell(J) \\ J \ \textup{is good} }}
\La T \Delta\ci I f , \Delta\ci J g \Ra \, d\bP(\omega)
\end{align*}
\end{lm}
\begin{proof}
It is more convenient to use probabilistic notation in the proof. 
Let 
\[
f_\gw := \sum_{\substack{I \in \cD_\omega \\ I \ \text{is good} }} \Delta\ci I f . 
\]
Then for any $f, g\in L^2$, 
\begin{align*}
\bE\ci\Omega \La  f_\gw , g\Ra\ & = \bE_\Omega \sum_{\substack{I \in \cD_\omega \\ I \ \text{is good} }} \La \Delta\ci I f, \Delta\ci I g \Ra \ 
\\
& = \sum_{k\in\Z} \bE_\Omega \bE_{\sA_k} \sum_{\substack{I \in \cD_\omega: \ell(I)=2^{k} \\ I \ \text{is good} }} \La \Delta\ci I f, \Delta\ci I g \Ra   \\
& = \sum_{k\in\Z} \bE_\Omega \bE_{\sA_k} \sum_{\substack{I \in \cD_\omega: \ell(I)=2^{k}  }} \La \Delta\ci I f, \Delta\ci I g \Ra \1\ci{
\{I \text{ is good}\}} (\om).
\end{align*}
To compute the conditional expectation let us notice that the position of the cubes $I\in\cD_\om$, $\ell(I) =  2^k$ depends only on the random  variables $\om_j$, $j<k$. On the other hand, the event that  a cube $I\in\cD_\om$, $\ell(I) =  2^k$ is good depends only on the variables $\om_j$, $j\ge k$, and for fixed variables $\om_j$, $j<k$ the corresponding conditional probability  of this event is $\pi\ti{good}$, so we can write for the conditional expectation
\begin{align}
\label{cond-exp}
\bE_{\sA_k}     \1\ci{\{I \text{ is good}\}} (\om) =\pi\ti{good}. 
\end{align}
 Therefore
\[
 \bE_{\sA_k} \sum_{\substack{I \in \cD_\omega: \ell(I)=2^{-k}  }} 
 \La \Delta\ci I f, \Delta\ci I g \Ra  \1\ci{
\{I \text{ is good}\}} (\om)
 = \pi\ti{good} \sum_{\substack{I \in \cD_\omega: \ell(I)=2^{-k}  }} 
 \La \Delta\ci I f, \Delta\ci I g \Ra , 
\]
which gives us 
\begin{align}
\label{E-fg-g}
\bE_\Omega \La  f_\gw , g\Ra\ =  \pi\ti{good} \La f, g \Ra. 
\end{align}

Applying this identity to $\La Tf_\gw, g\Ra = \La f_\gw, T^* g\Ra$ (with $T^*g$ instead of $g$) we get
\begin{align}
\notag
\pi\ti{good} \La T f, g \Ra & = \bE_\Omega \La  T f_\gw , g \Ra  \\
\notag
& = \bE_\Omega \sum_{\substack{I, J \in \cD_\omega \\ \ell(I) \le \ell(J) \\ I \ \text{is good} }}
\La T \Delta\ci I f , \Delta\ci J g \Ra  
+
\sum_{k\in\Z} \bE_\Omega \bE\ci{\sA_k} \sum_{\substack{I, J \in \cD_\omega \\ \ell(I)=2^k, \, \ell(I) > \ell(J)  }}
\La T \Delta\ci I f , \Delta\ci J g \Ra\,   \1\ci{\{I \text{ is good}\}}
\\
\label{split1}
 & =
\bE_\Omega \sum_{\substack{I, J \in \cD_\omega \\ \ell(I) \le \ell(J) \\ I \ \text{is good} }}
\La T \Delta\ci I f , \Delta\ci J g \Ra  
+
\pi\ti{good}
\bE_\Omega \sum_{\substack{I, J \in \cD_\omega \\ \ell(I) > \ell(J)  }}
\La T \Delta\ci I f , \Delta\ci J g \Ra  ;
\end{align}
here again in the last equality we used \eqref{cond-exp} and the fact that for $2^k=\ell(I) \ge \ell(J)$ the position of $I$ and $J$ depends on the variables $\om_j$, $j<k$, while the property of $I$ depends on the variables $\om_j$, $j\ge k$ and is not influenced by the position of $J$. 

\begin{rem}
\label{r:sum-int}
To justify the interchange of the summation and expectation $\bE_\Omega$ in \eqref{split1} we first observe  that for smooth $f$  
\begin{align*}
\|\Delta\ci I  f\|_\infty \le
\left\{
\begin{array}{ll} C(d) \|\nabla f\|_\infty \ell(I) \qquad & \ell(I) < 1\,, \\
\|f\|_\infty |I|^{-1}   & \ell(I)\ge 1\,. 
\end{array}
\right. 
\end{align*}
So, if we denote 
\begin{align*}
f^k_\om := \sum_{I\in \cD_\om: \ell(I) = 2^k} \Delta\ci I f, \qquad
f_\gw^k := \sum_{\substack{I \in \cD_\omega \,: \ell(I)=2^k\\ I \ \text{is good} }} \Delta\ci I f ,
\end{align*}
then, integrating the previous estimates we have for $f\in C^\infty_0$
\begin{align*}
\|f^k_\om \|_{L^2}, \ \|f_\gw^k\|_{L^2} \le 
C(f) \min\{ 2^k, 2^{-kd}\}\,,
\end{align*}
so 
\begin{align*}
\sum_{k\in\Z} \|f^k_\om \|_{L^2} \le C(f), \qquad \sum_{k\in\Z} \|f_\gw^k \|_{L^2} \le C(f) .
\end{align*}
Then for $f, g\in C^\infty_0$
\begin{align*}
\sum_{j,k\in\Z} |\La T f_\gw^k , g_\om^j\Ra |  \le \|T\| C(f) C(g), 
\end{align*}
which justifies the first interchange of summation and integration in \eqref{split1}. The same estimate holds if we replace $f_\gw^k$ by $f_\om^k$, and this justifies the second interchange. 

Note also that the sum $f_\om^k$ has at most $C(f, k)$ non-zero terms $\Delta\ci I f$ (where $C(f, k)<\infty$ does not depend on $\om$), so for   fixed $k$ and $j$ we can interchange summation over $I$, $\ell(I)=2^k$ and integration without any problems.  

\end{rem}

Let us continue with the proof of Lemma \ref{l:avg1}. 
Since for all $\omega\in\Omega$
\[
\La T f, g\Ra = \sum_{I, J\in\cD_\omega} \La T \Delta\ci I f, \Delta\ci J g\Ra , 
\]
averaging over all $\omega$ we get 
\begin{equation}
\label{split2}
\La T f, g\Ra = \bE_\Omega \sum_{\substack{I, J \in \cD_\omega \\ \ell(I) \le  \ell(J)  }}
\La T \Delta\ci I f , \Delta\ci J g \Ra 
+
\bE_\Omega \sum_{\substack{I, J \in \cD_\omega \\ \ell(I) > \ell(J)  }}
\La T \Delta\ci I f , \Delta\ci J g \Ra . 
\end{equation}
Multiplying this identity by $\pi\ti{good}$ and comparing with \eqref{split1} we get that 
\begin{equation}
\label{half1}
\pi\ti{good} \bE_\Omega \sum_{\substack{I, J \in \cD_\omega \\ \ell(I) \le  \ell(J)  }}
\La T \Delta\ci I f , \Delta\ci J g \Ra  = 
\bE_\Omega \sum_{\substack{I, J \in \cD_\omega \\ \ell(I) \le \ell(J) \\ I \ \text{is good} }}
\La T \Delta\ci I f , \Delta\ci J g \Ra . 
\end{equation}
\begin{rem*}
Note, that the above identity cannot be obtained by directly applying the above trick with the conditional expectation to the right side. If $2^s = \ell(I)<\ell(J)=2^k$, then the position of $I$ and $J$ is defined by the variables $\om_j$, $j<k$, and the property of $I$ being good depends on $\om_j$, $j\ge s$. Thus the conditional probability of $I$ being good depends on the mutual position of $I$ and $J$ and so  there is no splitting we used proving \eqref{E-fg-g}, \eqref{split1}.
\end{rem*}

We can repeat the reasoning  leading to \eqref{split2} without any changes to the splitting into $\ell(I) < \ell(J)$ and $\ell(I)\ge \ell(J)$ to get
\[
\pi\ti{good} \bE_\Omega \sum_{\substack{I, J \in \cD_\omega \\ \ell(I) <  \ell(J)  }}
\La T \Delta\ci I f , \Delta\ci J g \Ra  = 
\bE_\Omega \sum_{\substack{I, J \in \cD_\omega \\ \ell(I) < \ell(J) \\ I \ \text{is good} }}
\La T \Delta\ci I f , \Delta\ci J g \Ra  . 
\]
From the symmetry between $I$ and $J$ we can conclude that 
\begin{equation}
\label{half2}
\pi\ti{good} \bE_\Omega \sum_{\substack{I, J \in \cD_\omega \\ \ell(I) >  \ell(J)  }}
\La T \Delta\ci I f , \Delta\ci J g \Ra  = 
\bE_\Omega \sum_{\substack{I, J \in \cD_\omega \\ \ell(I) > \ell(J) \\ J \ \text{is good} }}
\La T \Delta\ci I f , \Delta\ci J g \Ra  .
\end{equation}

Substituting \eqref{split1} and \eqref{split2} into \eqref{half2} we get 
\begin{align}
\notag
\La T f, g\Ra & = \bE_\Omega \sum_{\substack{I, J \in \cD_\omega \\ \ell(I) \le  \ell(J)  }}
\La T \Delta\ci I f , \Delta\ci J g \Ra  
+
\bE_\Omega \sum_{\substack{I, J \in \cD_\omega \\ \ell(I) > \ell(J)  }}
\La T \Delta\ci I f , \Delta\ci J g \Ra  
\\
\notag
\label{good1}
& 
=\pi^{-1}\ti{good} \bE_\Omega \sum_{\substack{I, J \in \cD_\omega \\ \ell(I) \le \ell(J) \\ I \ \text{is good} }}
\La T \Delta\ci I f , \Delta\ci J g \Ra  
+
\pi^{-1}\ti{good}
\bE_\Omega \sum_{\substack{I, J \in \cD_\omega \\ \ell(I) > \ell(J) \\ J \ \text{is good} }}
\La T \Delta\ci I f , \Delta\ci J g \Ra 
\end{align}
\end{proof}

\subsection{Subtracting paraproducts}
For a \cz operator $T$ in $L^2(\R^d)$ and a dyadic latttice $\cD_\om$, define the dyadic paraproduct $ \Pi_T^\om$
\begin{align*}
\Pi\ci T^\om f := \sum_{Q\in\cD_\om} (\E\ci Q f) \Delta\ci Q T\1.
\end{align*}
Here $\Delta\ci Q T\1$ is defined by duality, 
\begin{align*}
\La \Delta\ci Q T \1, g\Ra: = \La   \1, T^* \Delta\ci Q g\Ra \qquad \forall g\in L^2;
\end{align*}
the right side here is well defined, as one can easily show that $T^* \Delta\ci Q g \in L^1$. (This is a pretty standard place in the theory of \cz operators.) 

Define operators $\wt T_\om$ 
\begin{align*}
\wt T_\om := T - \Pi_T^\om -(\Pi_{T^*}^\om)^*
\end{align*}

\begin{rem}
The matrix of the paraproduct $\Pi^\om_T$ has a very special ``triangular'' form. Namely, a block $\Delta\ci R \Pi_T^\om \Delta\ci Q$, $Q, R\in\cD_\om$ can be non-zero only if $R\subsetneqq Q$. Notice also, that if $\ell(Q)=2^k$, then the block $\Delta\ci R \Pi_T^\om \Delta\ci Q$ does not depend on the variables $\om_j$, $j\ge k$.

From the above observation  is easy to see that if $Q, R\in\cD_\om$, $\max\{\ell(Q), \ell(R)\} =2^k$, then the block $\Delta \ci R \wt T_\om \Delta\ci Q$ does not depend on variables $\om_j$, $j\ge k$, and that 
\[
\Delta \ci R \wt T_\om \Delta\ci Q = \Delta \ci R  T \Delta\ci Q
\] 
if $Q\cap R = \varnothing$ or $Q = R$. 
\end{rem}

The  paraproducts were introduced in \cz theory in the proofs of $T(1)$ and $T(b)$ theorems. The main idea is that one can estimate the operators $\wt T_\om$ by estimating the absolute values of the entries of its matrix in the Haar basis, but one cannot, in general, do the same with paraproducts (and so with a general \cz operator $T$). The papraproducts, however can be easily estimated by the Carleson Embedding Theorem, using the condition $T1\in \text{BMO}$ ($Tb\in\text{BMO}$).  

\begin{df*}
Let $D(Q, R)$ be the so-called \emph{long distance} between the cubes $Q$ and $R$, see \cite{NTV5}, 
\[
D(Q, R):= \dist (Q, R) + \ell(Q) + \ell(R). 
\]
\end{df*}

\begin{lm}
\label{l:ThQ-hR}
Let $T$ be a \cz operator (with parameter $\alpha$), and let $Q, R\in\cD_\om$, $\ell(Q) \le \ell(R)$. Let $h\ci Q$ and $h\ci R$ be Haar functions, $\|h\ci Q\|=\|h\ci R\|=1$. If $Q$ is a good cube, then  
\begin{align*}
|\La \wt T_\om h\ci Q, h\ci R \Ra |, \, |\La \wt T_\om h\ci R, h\ci Q \Ra | \le C \frac{\ell(Q)^{\alpha/2}\ell(R)^{\alpha/2}}{D(Q, R)^{d+\alpha}} |Q|^{1/2} |R|^{1/2}, 
\end{align*}
where $C=C(r_0, d, \alpha, C\ti{cz})<\infty$. 
\end{lm}


The proof is pretty standard, see \cite{NTV5} for example. 

\begin{lm}
Let $C=C(r_0, d, \alpha, C\ti{cz})\ne 0$ be the constant from the above Lemma \ref{l:ThQ-hR}, and let $|a\ci{Q,R}|\le 1$. Then for any
 dyadic lattice $\cD_\om$ and for any $m, n\in\Z_+$, $m\ge n$ the operators 
\begin{align*}
C^{-1} \sum_{M\in\cD_\om} \  & \sum_{\substack{ Q, R\in \cD_\om : \, Q, R \subset M\\  \ell(Q) =2^{-m} \ell(M) \\ \ell(R) =2^{-n}\ell(M)\\ Q \textup{ is good}      }}  
a\ci{Q,R} 2^{ (m+n) \alpha/2} \,\cdot  \frac{D(Q, R)^{d+\alpha}}{\ell(M)^{d+\alpha}} \, \Delta\ci R \wt T_\om \Delta\ci Q
\end{align*}
is a dyadic shift with parameters $m$, $n$, and the same holds if we replace $ \Delta\ci R \wt T_\om \Delta\ci Q$ by $\Delta\ci Q \wt T_\om \Delta\ci R$.
\end{lm}

\begin{proof}
We will need the notion of the \emph{standard Haar basis} here. For an interval $I\subset \R$ let $h^0\ci I := |I|^{-1/2} \1\ci I$, and let $h^1\ci I$ be the standard $L^2$-normalized Haar function, 
\begin{align*}
h^1\ci I := |I|^{-1/2} ( \1\ci{I_+} - \1\ci{I_-} ), 
\end{align*}
where $I_+$ and $I_-$ are the right and the left halves of $I$ respectively.

For a cube $Q=I_1\times I_2 \times\ldots \times I_d \in\R^d$ and an index $j$, $0\le j <2^d$,   let 
\[
h\ci Q^j(x) := \prod_{k=1}^d h\ci{I_k}^{j_k}(x_k), \qquad x=(x_1, x_2, \ldots, x_d),  
\]
where $j_k\in\{0,1\}$ are the coefficients in the binary decomposition $j=\sum_{k=1}^d j_k 2^{k-1}$ of $j$. 

The system $h^j\ci Q$, $j=1, \ldots, 2^d -1$ form an orthonormal basis in $\Delta\ci Q L^2$, which we will call the \emph{standard Haar basis}. 

Note that $h^0\ci Q = |Q|^{-1/2} \1\ci Q$.

The block $\Delta\ci R \wt T_\om \Delta\ci Q$ can be represented as 
\[
\Delta\ci R \wt T_\om \Delta\ci Q =\sum_{j,k=1}^{2^d -1} c_{j,k}(Q,R)\La \fdot , h^k\ci Q \Ra h^j\ci R
\]
where $c_{j,k}(Q, R) = \La \wt T_\om h^k\ci Q, h^j\ci R \Ra$. 

Since $\|h^j\ci Q\|_\infty = |Q|^{-1/2}$ we can estimate using Lemma \ref{l:ThQ-hR}
\begin{align}
\label{bd1}
|c_{j,k}(Q, R)| \cdot \|h^k\ci Q\|_\infty \cdot \| h^j\ci R\|_\infty \le 
\ C \frac{\ell(Q)^{\alpha/2}\ell(R)^{\alpha/2}}{D(Q, R)^{d+\alpha}} ,
\end{align}
where $C = C(r_0, d, \alpha, C\ti{cz})$ is the constant from Lemma \ref{l:ThQ-hR}. 

Clearly for fixed $j, k$ and the constant $C$ from Lemma \ref{l:ThQ-hR} we can write 
\begin{align*}
 C^{-1}\sum_{M\in\cD_\om} \  & \sum_{\substack{ Q, R\in \cD_\om :  \, Q, R \subset M\\  \ell(Q) =2^{-m} \ell(M) \\ \ell(R) =2^{-n}\ell(M)\\ Q \textup{ is good}      }}  
a\ci{Q, R}  2^{ (m+n) \alpha/2} \,\cdot  \frac{D(Q, R)^{d+\alpha}}{\ell(M)^{d+\alpha}} 
c_{j,k}(Q,R)\La \fdot , h^k\ci Q \Ra h^j\ci R 
\\
& = 
\sum_{M\in\cD_\om}  \  \sum_{\substack{ Q, R\in \cD_\om : \, Q, R \subset M \\  \ell(Q) =2^{-m} \ell(M) \\ \ell(R) =2^{-n}\ell(M)\\ Q \textup{ is good}      }}  
\La \fdot , h\ci Q \Ra h\ci R
\end{align*}
where $h\ci Q$ and $h\ci R$ are multiples  of $h^k\ci Q$ and $h^j\ci R$. This sum has the structure of an elementary dyadic shift, and to prove the lemma we only need to estimate $\|h\ci Q\|_\infty \|h\ci R\|_\infty$.

Using \eqref{bd1} we get for  fixed cubes $Q$ and $R$
\begin{align*}
\|h\ci Q\|_\infty \|h\ci R\|_\infty 
& \le  \frac{\ell(Q)^{\alpha/2}\ell(R)^{\alpha/2}}{D(Q, R)^{d+\alpha}} 
2^{ (m+n) \alpha/2} \,\cdot  \frac{D(Q, R)^{d+\alpha}}{\ell(M)^{d+\alpha}} \\
& =  \frac{1}{\ell(M)^d } \cdot\frac{\ell(Q)^{\alpha/2}\ell(R)^{\alpha/2}}{\ell(M)^{\alpha}} 
2^{ (m+n) \alpha/2}
= \frac{1}{\ell(M)^d }, 
\end{align*}
because $\ell(Q)/\ell(M) = 2^{-m}$, $\ell(R)/\ell(M) =2^{-n}$. 


So, the above sum is indeed an elementary dyadic shift with parameters $m$, $n$. Summing over all $j, k$ we get the conclusion   of the lemma
\end{proof}

\subsection{Proof of Theorem \ref{T=avg-shift}}
As we explained before, see Lemma \ref{l:avg1}, we can represent $T$ as the average
\[
T = \pi\ti{good}^{-1} \bE_\Omega \sum_{\substack{ Q, R \in \cD_\om \\ \ell(Q) \le \ell(R) \\ Q \text{ is good}}} \Delta\ci R T \Delta\ci Q + 
\pi\ti{good}^{-1} \bE_\Omega \sum_{\substack{ Q, R \in \cD_\om \\ \ell(R) < \ell(Q) \\ R \text{ is good}}} \Delta\ci R T \Delta\ci Q; 
\] 
here and below in this section the averages $\bE_\Omega$ are understood in the weak sense, as equalities of the bilinear forms for $f, g\in C^\infty_0$.  As it was explained before in the proof of   Lemma \ref{l:avg1},   see Remark \ref{r:sum-int} there, in this case we can freely interchange the summation and expectation (integration) $\bE_\Omega$.  

Recalling the decomposition 
\[
T = \wt T_\om + \Pi^\om_T +(\Pi_{T^*}^\om)^*, 
\]
and using the fact that 
for $Q, R\in\cD_\om$
\[
\Delta\ci R \Pi^\om_T \Delta\ci Q = 0, \qquad \Delta\ci Q (\Pi^\om_{T^*})^* \Delta\ci R = 0
\] 
if $\ell(Q)\le \ell(R)$, 
we can write
\begin{align}
\label{decomp-para2}
T =  \pi\ti{good}^{-1} & \bE_\Omega   \sum_{\substack{ Q, R \in \cD_\om \\ \ell(Q) \le \ell(R) \\ Q \text{ is good}}} \Delta\ci R \wt T_\om \Delta\ci Q + 
\pi\ti{good}^{-1} \bE_\Omega \sum_{\substack{ Q, R \in \cD_\om \\ \ell(R) < \ell(Q) \\ R \text{ is good}}} \Delta\ci R \wt T_\om \Delta\ci Q \\
\notag
& +\pi\ti{good}^{-1} \bE_\Omega \sum_{\substack{ Q, R \in \cD_\om \\ \ell(Q) \le \ell(R) \\ Q \text{ is good}}} \Delta\ci R (\Pi^\om_{T^*})^* \Delta\ci Q + 
\pi\ti{good}^{-1} \bE_\Omega \sum_{\substack{ Q, R \in \cD_\om \\ \ell(R) < \ell(Q) \\ R \text{ is good}}} \Delta\ci R \Pi_T^\om \Delta\ci Q. 
\end{align}

\begin{lm}
\label{l:av-para}
For the paraproducts $\Pi^\om_T$ 
\begin{align*}
\bE_\Omega \sum_{\substack{ Q, R \in \cD_\om \\ \ell(R) < \ell(Q) \\ R \text{ is good}}} \Delta\ci R \Pi_T^\om \Delta\ci Q =
\bE_\Omega \sum_{\substack{ Q, R \in \cD_\om \\ \ell(R) \le \ell(Q) \\ R \text{ is good}}} \Delta\ci R \Pi_T^\om \Delta\ci Q 
= \pi\ti{good} \bE_\Om \Pi_T^\om
\end{align*}
\end{lm}

\begin{proof}
It is not hard to see from the definition of the paraproduct that for $f\in L^2$
\begin{align*}
\sum_{\substack{ Q, R \in \cD_\om \\ \ell(R) < \ell(Q) \\ R \text{ is good}}} \Delta\ci R \Pi_T^\om \Delta\ci Q f =
\sum_{\substack{ Q, R \in \cD_\om \\ \ell(R) \le \ell(Q) \\ R \text{ is good}}} \Delta\ci R \Pi_T^\om \Delta\ci Q f
=
\sum_{\substack{  R \in \cD_\om \\  R \text{ is good}}} (\Delta\ci R T\1) \E\ci R f.
\end{align*}
Applying $\bE_\Omega$ we get that 
\begin{align*}
\bE_\Omega \sum_{\substack{  R \in \cD_\om \\  R \text{ is good}}} (\Delta\ci R T\1) \E\ci R f
&  =
\sum_{k\in\Z} \bE_\Omega \bE_{\sA_k} \sum_{\substack{R\in\cD_\om \\ \ell(R) = 2^k }} (\Delta\ci R T\1) (\E\ci R f )\1\ci{R \text{ is good}} (\om)
\\
& = \pi\ti{good}\sum_{k\in\Z} \bE_\Omega  \sum_{\substack{R\in\cD_\om \\ \ell(R) = 2^k }} (\Delta\ci RT\1) \E\ci R f
= 
\pi\ti{good} \bE_\Omega \Pi^\om_T f; 
\end{align*}
here we again used the fact that by \eqref{cond-exp} $\bE_{\sA_k} \1\ci{R \text{ is good}} (\om)= \pi\ti{good}$ for $R\in\cD_\om$, $\ell(R)=2^k$. 
\end{proof}


By Lemma \ref{l:av-para} the second line in \eqref{decomp-para2} is $\bE_\Omega (\Pi_T^\om + (\Pi_{T^*}^\om)^*)$. We know that the paraproducts $\Pi_T^\om$ and $(\Pi_{T^*}^\om)^*$ are (up to a  constant factor $C=C(\alpha, d, C\ti{cz}, \|T\|)$) generalized dyadic shifts with parameters $0, 1$ and $1, 0$ respectively. 

So to prove the theorem we need to represent the first line in \eqref{decomp-para2} as the average of dyadic shifts. 
Let us represent the first term. 
For $m, n\in\Z_+$, $m\ge n$, define the dyadic shifts $\sha^\om_{m,n}$ as 
\begin{align*}
\sha^\om_{m,n} = \sum_{M\in\cD_\om} \sum_{\substack{Q, R\in \cD_\om :\, Q, R \subset M\\ \ell(Q) =2^{-m} \ell(M), \,\ell(R) = 2^{-n} \ell(M)\\ Q \text{ is good} }} 
\pi(Q|R)\cdot \rho^{-1}\ci{Q,R} \cdot 2^{ (m+n) \alpha/2} \,\cdot  \frac{D(Q, R)^{d+\alpha}}{\ell(M)^{d+\alpha}} \, \Delta\ci R \wt T_\om \Delta\ci Q,  
\end{align*}
where 
\[
\pi(Q|R) =\bP\{Q \text{ is good}|R\} =\bE_{\Omega[R]} \1\ci{Q \text{ is good}}
\] 
(note that $\ell(Q)\le \ell(R)$). The weights $\rho\ci{Q,R}$, $Q, R\in\cD_\om$, are defined by
\begin{align}
\label{eq:rho_QR}
\rho\ci{Q,R} := \bE_{\Omega[ R ] } \sum_{\substack{ M\in \cD_\om:\, Q, R \subset M  }}     \frac{D(Q, R)^{d+\alpha}}{\ell(M)^{d+\alpha}} \cdot
\1\ci{  Q \text{ is good} }(\om) ; 
\end{align}
note that in the above expression we assume (can assume) that the variables $\om_j$, $j<k$, determining the position of $R$ (and so of $Q$) are fixed. 

\begin{rem}
In general, $\rho\ci{Q,R}$ can be zero. However, it is not hard to see that $\rho\ci{Q, R}>0$ if $\pi(Q|R)>0$, so the dyadic shifts $\sha^\om_{m,n}$ are well defined. 
\end{rem}

Averaging we get 
\begin{align*}
\bE_\Omega & \sum_{\substack{m,n\in \Z :\, m\ge n}}  2^{-(m+n)\alpha/2} \sha^\om_{m, n}  
\\
 &= 
\bE_\Omega  \sum_{ \substack{Q, R\in\cD_\om \\ \ell(Q) \le \ell(R) \\ \pi(Q|R)\ne 0 }} 
\sum_{ \substack{M\in\cD_\om \\ Q, R\subset M }} \pi(Q|R)\cdot \rho^{-1}\ci{Q, R}  \frac{D(Q, R)^{d+\alpha}}{\ell(M)^{d+\alpha}}    \cdot
\1\ci{  Q \text{ is good} }(\om)   \Delta\ci R \wt T_\om \Delta\ci Q
\\    & = \bE_\Omega  \sum_{ \substack{Q, R\in\cD_\om \\ \ell(Q) \le \ell(R) \\ \pi(Q|R)\ne 0 }} \bE_{\Omega[ R]}
\pi(Q|R)\cdot \rho^{-1}\ci{Q, R}\cdot \Delta\ci R \wt T_\om \Delta\ci Q 
\sum_{ \substack{M\in\cD_\om \\ Q, R\subset M }}  
 \frac{D(Q, R)^{d+\alpha}}{\ell(M)^{d+\alpha}} \cdot
\1\ci{  Q \text{ is good} }(\om)  
\end{align*}
and recalling the definition of $\rho\ci{Q,R}$ we conclude 
\begin{align*}
\bE_\Omega & \sum_{\substack{m,n\in \Z :\, m\ge n}}  2^{-(m+n)\alpha/2} \sha^\om_{m, n} 
=
\bE_\Omega  \sum_{ \substack{Q, R\in\cD_\om \\ \ell(Q) \le \ell(R) }}  
\pi(Q|R) 
  \Delta\ci R \wt T_\om \Delta\ci Q.
\end{align*}

On the other hand
\begin{align*}
\bE_\Omega  \sum_{ \substack{Q, R\in\cD_\om \\ \ell(Q) \le \ell(R)\\ Q\text{ is good} }}  
 \Delta\ci R \wt T_\om \Delta\ci Q 
 &=
 \sum_{k\in\Z} \bE_\Omega \bE_{\sA_k}
 \sum_{ \substack{Q, R\in\cD_\om \\ \ell(Q) \le \ell(R)=2^k } }  \1\ci{Q\text{ is good} }(\om) \cdot \Delta\ci R \wt T_\om \Delta\ci Q  \\
 & = \sum_{k\in\Z} \bE_\Omega     
 \sum_{ \substack{Q, R\in\cD_\om \\ \ell(Q) \le \ell(R)=2^k } } (\bE_{\Omega[R]} \1\ci{Q\text{ is good} })  \Delta\ci R \wt T_\om \Delta\ci Q \\
 & = \bE_\Omega \sum_{ \substack{Q, R\in\cD_\om \\ \ell(Q) \le \ell(R) } } \pi(Q|R) \Delta\ci R \wt T_\om \Delta\ci Q ,
\end{align*}
so
\begin{align*}
\bE_\Omega & \sum_{\substack{m,n\in \Z :\, m\ge n}}  2^{-(m+n)\alpha/2} \sha^\om_{m, n} 
=
\bE_\Omega  \sum_{ \substack{Q, R\in\cD_\om \\ \ell(Q) \le \ell(R)\\ Q\text{ is good} }}  
 \Delta\ci R \wt T_\om \Delta\ci Q .
\end{align*}

It now remains to show that $\sha^\om_{m, n}$ are (up to a constant factor) are the dyadic shifts. The operators $\sha^\om_{m, n}$ have the appropriate structure, so we only need to prove the estimates, i.e. to prove that the weights $\rho\ci{Q, R}$ are uniformly bounded away from $0$. The necessary estimate follows from Lemma \ref{l:rho_P_Q} below. 

So, we have decomposed the first term in \eqref{decomp-para2} as the average of dyadic shifts. The decomposition of the second term is carried out similarly, so Theorem \ref{T=avg-shift} is proved (modulo Lemma \ref{l:rho_P_Q}). \hfill \qed

\begin{lm}
\label{l:rho_P_Q}
Let $Q, R\in\cD_\om$, $\ell(Q) \le \ell(R)$. Then 
\begin{enumerate}
    \item $\pi(Q|R) > 0$ if and only if $Q$ is ``good up to the level of $R$'', meaning that 
    \begin{equation}
    \label{eq:good-to_level_R}
    \dist (Q, Q') \ge \ell(Q)^\gamma \ell(Q')^{1-\gamma} \qquad \forall Q'\in \cD_\om:\ 2^{r_0} \ell(Q) < \ell(Q') \le \ell(R);
    \end{equation}
    note that the cubes $Q'$ do not depend on the variables $\om_j$, $j\ge k$ where $2^k=\ell(R)$. 
    
    \item There exists a constant $c=c(d, r_0, \gamma)$ such that 
    \[
    \rho\ci{Q,R} \ge c(d, r_0)  \qquad \forall Q, R\in \cD_\om :  \ \pi(Q|R)\ne 0 .  
    \] 
\end{enumerate}
\end{lm}

\begin{proof}
We want to estimate conditional probability end expectation with $R$ and $Q$ fixed. That means the lattice up to the level of $R$ is fixed, so nothing changes if we replace $R$ by a cube in the same level. So, without loss of generality we can assume that $Q\subset R$. 

Let us first consider a special case. Let $\ell(R) = \ell(Q) 2^{s}$, where 
\begin{align}
    \label{eq:k_ge}
    s\ge 2/\gamma + r_0 \cdot (1-\gamma)/\gamma,  
\end{align}
and let 
\[
\dist(Q, \partial R) \ge \frac14 \ell(R). 
\]
Then the estimate \eqref{eq:k_ge} implies that 
\begin{align*}
\ell(Q)^\gamma \bigl[ 2^{r_0} \ell(R) \bigr]^{1-\gamma}
=2^{-s\gamma}2^{r_0(1-\gamma)}\ell(R)\le \frac14 \ell(R), 
\end{align*}
meaning that for any cube $M\in\cD_\om$, $\ell(R) \le \ell(M) \le 2^{r_0} \ell(R) $ (assuming that the lattice $\cD_\om$ is fixed up to the level of $R$)
\begin{align}
\notag
\ell(Q)^\gamma \ell(M)^{1-\gamma}  \le \frac14 \ell(R)  & \le \dist(Q, \partial R)
\\
\label{eq:Q-M_good-1}
&\le \dist (Q, \partial M). 
\end{align}
On the other hand, if  $\ell(M)> 2^{r_0}\ell (R)$ and the pair $R$, $M$ is good, meaning that 
\[
\dist (R, \partial M) \ge \ell(R)^\gamma \ell(M)^{1-\gamma} 
\]
then 
\begin{align}
\label{eq:Q-M_good}
\dist(Q,\partial M) \ge \ell(Q)^\gamma \ell(M)^{1-\gamma},  
\end{align}
so the pair $Q$, $M$ is also good. 

Therefore, if the cube $R$ is good, then $Q$ is good as well: as we just discussed, the inequality \eqref{eq:Q-M_good} holds if $\ell(M) > 2^{r_0}\ell (R)$, and it holds for $\ell(R)\le \ell(M) \le 2^{r_0} \ell(R)$ by \eqref{eq:Q-M_good-1}. And the assumption   \eqref{eq:good-to_level_R} covers the remaining cases. 

So,  in our special case $\pi(Q|R) \ge \pi\ti{good}$. 

The general case can be easily reduced to this special situation. Namely, if $Q\subsetneq R$, then with probability at least $2^{-d}$ the parent $\wt R$ of $R$ satisfies 
\[
\dist(Q, \partial \wt R) \ge \frac14 \ell(\wt R);
\]
one can easily see that for $d=1$, and considering the coordinates independently, one gets the conclusion. 

Applying this procedure $s_0-1$ times, where $s_0$ is the smallest integer satisfying \eqref{eq:k_ge}, we arrive (with probability at least $2^{-(s_0-1)d}$) to the special  situation we just discussed. Therefore for $Q\subsetneq R$ (equivalently $\ell(Q) <\ell(R)$) statement \cond1 is proved with the estimate 
\begin{align}
\label{eq:est_pi(Q|R)}
\pi(Q|R) \ge 2^{-(s_0-1)d} \pi\ti{good} =: \pi_0. 
\end{align}
Finally, if $Q=R$, we with probability $1$ arrive to the previous situation, so the statement \cond1 is now completely proved with estimate \eqref{eq:est_pi(Q|R)}. 

The statement \cond2 is now easy. First note, that if  $\tau \in \Z$ is such that $2^\tau> D(Q, R)$, then 
\begin{equation}
\label{eq:prob_Q_Q_subset_M}
\bP\{ \exists M \in\cD_\om : \ell(M) =2^\tau, \ Q, R\subset M\ |\, R\} \ge 1 - d \cdot 2 D(Q, R) /2^\tau. 
\end{equation}
Indeed, in one dimension the probability that such $M$ does not exists can be estimated above by  $2 D(Q, R) /2^\tau$, so to get the estimate of non existence in $\R^d$ we can just multiply it by $d$. 
The extra factor $2$ appears in one dimensional case   because  $M$ cannot be moved  continuously, but only in multiples of $\ell(R)$. 

Define 
\[
\tau_0:= \lfloor \log_2 (d D(Q, R) /\pi_0 ) \rfloor + 3, 
\]
so 
\[
d \cdot 2 D(Q, R) /2^{\tau_0} \le \pi_0/2. 
\]

Comparing the estimates  \eqref{eq:est_pi(Q|R)} and \eqref{eq:prob_Q_Q_subset_M} of probabilities, we can get that for fixed $Q$ and $R$ the probability that $Q$ is good and that $Q, R\subset M$ for some $M\in \cD_\om$, $\ell(M) = 2^{\tau_0}$, is at least $\pi_0/2$.

On the other hand, the definition of $\tau_0$ implies that 
$
\ell(M) =2^{\tau_0}  \le 8\cdot d\cdot D(Q,R)/\pi_0  
$, 
so 
\[
 D(Q, R) /\ell(M) \ge \pi_0/8. 
\]
Therefore, the contribution to the sum \eqref{eq:rho_QR} defining $\rho\ci{Q, R}$ of the term with such $M$ alone is at least 
\[
\left( \pi_0/8 \right)^{d+\alpha} \pi_0/2. 
\]
That proves \cond2 and so the lemma. 
\end{proof}


\section{Sharp weighted estimate of dyadic shifts}

Recall, that for a dyadic shift $\sha$ with parameters $m$ and $n$ its complexity  is $r:=\max(m,n)$. 
In this section we assume that a dyadic lattice $\cD$ is fixed. 
Let  $\sha$ be an elementary (possibly generalized) dyadic shift 
\begin{equation}
\label{sha-1}
\sha f(x) = \sum_{Q\in \cD} \int_Q a\ci Q (x, y) f(y) dy 
\end{equation}
where $a\ci Q$ are supported on $Q\times Q$, $\|a\ci Q\|_\infty \le |Q|^{-1}$ (in this section we will  incorporate $|Q|^{-1}$ into $a\ci Q$). Let  $\cA\subset \cD$  be a collection of dyadic cubes. Define the restricted dyadic shift $\sha\ci\cA$ by taking the sum in \eqref{sha-1} only over $Q\in\cA$. 

As it was shown by Theorem \ref{T=avg-shift} that a \cz operator $T$ is a weighted average of dyadic shifts with exponentially (in complexity of shifts) decaying weights, to prove Theorem \ref{A2}  it is sufficient to get an estimate of the norm of dyadic shifts which is polynomial in complexity. The following theorem, indeed, achieves a norm bound which is \emph{quadratic} in complexity. This is the second new main result of this paper and represents a substantial quantitative improvement over earlier sharp weighted bounds for dyadic shifts \cite{LPR,CUMP1}, which were exponential in complexity. Note that the paper \cite{H}, while using dyadic shifts as auxiliary operators in the original proof of Theorem~\ref{A2}, circumvented the question of actually estimating their norm. This is achieved in \cite{H} by going through the test conditions of rather involved paper \cite{PTV1}.
 
\begin{thm}
\label{t:sharp-shift-A2}
Let $\sha$ be an elementary (possibly generalized) dyadic shift of complexity $r$ in $\R^d$, such that all  restricted shifts $\sha\ci \cA$  are uniformly  bounded in $L^2$
\begin{equation}
\label{subshift-bd}
\sup_{\cA\subset \cD} \| \sha\ci\cA\|\ci{L^2\to L^2} =:B_2= B\ci\sha <\infty .
\end{equation}
Then for any $A_2$ weight $w$
\begin{equation}
\label{pol-A2-1}
\| \sha f\|\ci{L^2(w)} \le C 2^{3d/2} (r+1)^2 \left( B_2^2 + 1 \right) [w]\ci{A_2} \|f\|\ci{L^2(w)} , \qquad \forall f\in L^2 (w)
\end{equation}
where  $C$ is an absolute constant. 
\end{thm}

Note that for dyadic shifts we are considering (that is non-generalized dyadic shifts and paraproducts), the assumption about uniform boundedness of $\sha\ci \cA $ is satisfied automatically. Namely, any non-generalized dyadic shift is a contraction in $L^2$, so \eqref{subshift-bd} holds with $B=1$. It is also easy to see that for the paraproducts $\|\sha\ci \cA\|\ci{L^2\to L^2} \le \|\sha \|\ci{L^2\to L^2}$.

The  estimate \eqref{pol-A2-1} with $C$ depending exponentially on $r$ was proved (for non-generalized dyadic shifts) in \cite{LPR}. However, careful analysis of proofs there allows (after some modifications)  to obtain polynomial estimates.

Compared to \cite{LPR}, the main new ingredients here are:
\begin{itemize}
    \item The sharp two weight estimate of Haar shifts, see above Theorem \ref{sh2w}, which is essentially the main result of \cite{NTV6} (with the additional assumptions about ``size'' of the operator), with the dependence of the estimates on all parameters spelled out. 
    \item  Proposition 5.1 of \cite{H}, reproduced as Theorem \ref{LPRnW} below, which gives \emph{linear} in complexity of $\sha$ estimate of the unweighted weak $L^1$ norm of $\sha$; the corresponding estimate in \cite{LPR} was exponential in complexity.   
\end{itemize}

Replacing $f$ in \eqref{pol-A2-1} by $fw^{-1}$ and noticing that $\|fw^{-1}\|\ci{L^2(w)} = \|f\|\ci{L^2(w^{-1})}$ we can rewrite it as 
\begin{equation}
\label{pol-A2-2}
\| \sha (fw^{-1}) \|\ci{L^2(w)} \le C 2^{3d/2} (r+1)^2  \left( B_2^2 + 1 \right)     [w]\ci{A_2} \|f\|\ci{L^2(w^{-1})}\,, \qquad \forall f\in L^2(w^{-1}),  
\end{equation}
so we are in the settings of Theorem \ref{sh2w} with $d\mu=w^{-1} dx$, $d\nu = w dx$. By Theorem \ref{sh2w}, to prove estimate \eqref{pol-A2-2} is is sufficient to show that 
\begin{align}
\notag
\int_Q |\sha (\1\ci Q w^{-1})|^2 w dx  & \le B [w]\ci{A_2}^2 w^{-1}(Q), \qquad \forall f\in L^2(w^{-1}) \\
\label{A2-test-1}
\int_Q |\sha (\1\ci Q w)|^2 w^{-1} dx   &\le B [w]\ci{A_2}^2 w(Q), \qquad \forall f\in L^2(w) 
\end{align}
where  
\[
B^{1/2} = C 2^{d} (r+1)  \left( B_2^2 + 1 \right) 
\] 
with an absolute constant $C$.

Since $[w^{-1}]\ci{A_2}=[w]\ci{A_2}$, one can get one estimate from the other by replacing $w$ by $w^{-1}$. Thus, to prove Theorem \ref{t:sharp-shift-A2} and so the main result (Theorem \ref{A2}) we only need to prove one of the above estimates, for example \eqref{A2-test-1}. 

The rest of the section is devoted to proving \eqref{A2-test-1}

\subsection{Weak type estimates for dyadic shifts}

Let $\|\sha\|_2$ be a shorthand for $\|\sha\|\ci{L^2\to L^2}$. 
We say that a shift $\sha$ has \emph{scales separated by $r$ levels}, if all cubes $Q$ with $a_Q\not\equiv 0$ in \eqref{sha-1} satisfy $\log_2\ell(Q)\equiv j\mod r$ for some fixed $j\in\{0,1,\ldots,r-1\}$.

The following result reproduces Proposition~5.1 of \cite{H} with an additional observation concerning shifts which have their scales separated. This seemingly technical variant allows us to obtain the asserted quadratic, rather than cubic, dependence on complexity in Theorem~\ref{t:sharp-shift-A2}.

\begin{thm}
\label{LPRnW}
Let $\sha$ be a generalized elementary dyadic shift with parameters $m,n$. 
Then $\sha$ has weak type $1$-$1$ with the estimate
\begin{equation}
\label{LPR2nW}
\|\sha \|\ci{L^1\to L^{1,\infty}} \le C(d, m, \|\sha\|_2) = 2^{d+2}\|\sha\|_2^2 + 1 + 4m, 
\end{equation}
 meaning that for all $f\in L^1$ and for all $\la>0$
\[
\left| \left\{ x: |\sha f (x)|>\lambda \right\}\right| \le \frac{C(d, m, \|\sha\|_2)}{ \la } \|f\|_1. 
\]
If $\sha$ has scales separated by $r\geq m$ levels, then we have the improved estimate
\begin{equation*}
  \|\sha\|_{L^1\to L^{1,\infty}}\leq C(d,1,\|\sha\|_2)=2^{d+2}\|S\|_2^2+5.
\end{equation*}
\end{thm}

\begin{proof}

Our shift $\sha$  can be written (see \eqref{sha1}) as 
\[
\sha f(x) =  \sum_{Q\in \cD}  \int  a\ci Q(x,y) f(y) dy\,, 
\]
where $a\ci Q$ is supported on $Q\times Q$ and $\|a\ci Q\|_\infty \le |Q|^{-1}$ (we incorporated the factor $|Q|^{-1}$ from \eqref{sha1} into $a\ci Q$ here). 
It follows from the representation \eqref{sha-aQ} of $a\ci Q$ that for fixed $x$ the function $a\ci Q(x, \fdot)$ is  constant on cubes $Q'\in\cD$, $\ell(Q')< 2^{-m} \ell(Q)$.

%

To estimate its weak norm we use the standard \cz decomposition at height $\lambda>0$ with respect to the  dyadic lattice $\cD$. 
Namely, as it is well known, see for example \cite[p.~286]{Grafakos_CFA-book_2008}, given $f\in L^1$ there exists a decomposition $f=g+b$, $b=\sum_{Q\in\cQ} b\ci Q$,  where $\cQ\subset \cD$ is a collection of disjoint dyadic cubes, such that
\begin{enumerate}
    \item $\|g\|_1\le  \|f\|_1$, $\|g\|_{\infty} \le 2^d\,\lambda$.
    \item Each function $b_Q$ is supported on a cube $Q$ and
    \[
    \|b\ci Q\|_1 \le  2\cdot\|\1_Q f\|_1,  
    \qquad \int_{\R^d} b\ci Q \,dx =0.
    \]
    \item $\sum_{Q\in\cQ} |Q| \le \lambda^{-1} \|f\|_1$. 
\end{enumerate}

The property \cond1 of the \cz decomposition implies that 
\begin{align}
    \label{CZd-f-L2}
    \|f\|_2^2 \le 2^d \lambda \|f\|_1
\end{align}

As usual, we can estimate
\begin{align*}
|\{x: \sha f(x)|>\lambda\} |  \le |\{x: |\sha g(x)|>\lambda/2\} | + |\{x: |\sha b (x)|>\lambda/2\} |
\end{align*}
(one of the two terms should be at least half of the sum). The measure of the first set is estimated using the boundedness of $\sha$ in $L^2$
\begin{align*}
|\{x: |\sha g(x)|>\lambda/2\} | \le \|\sha\|_2^2 \|g\|_2^2 \frac{4}{\lambda^2} \le \|\sha\|_2^2 \frac{2^{d+2}}{\lambda}\|f\|_1 ,
\end{align*}
where $\|\sha\|_2$ is the shorthand for $\|\sha\|\ci{L^2\to L^2}$;  we used \eqref{CZd-f-L2} to get the second inequality.

 To estimate $ |\{x: |\sha b(x)|>\lambda/2\} |$ we fix a $Q\in \mathcal{Q}$ and write a pointwise inequality:
\[
|\sha b_Q(x)| \le  \sum_{R\in \cD:\, Q \subsetneqq R} \biggl|\int_R a\ci R(x,y) b\ci Q(y)dy \biggr| + \biggl|\sum_{R\in \cD:\, R \subset Q} \int_R a\ci R(x,y)b\ci Q(y)dy \biggr|\,.
\]
Therefore, summing in $Q\in \mathcal{Q}$, we get
\begin{align*}
|\sha b(x)| 
& \le \sum_{Q\in\cQ} \sum_{R\in \cD:\, Q \subsetneqq R} \biggl|\int_R a\ci R(x,y) b\ci Q(y)dy \biggr| +
\sum_{Q\in\cQ} \biggl|\sum_{R\in \cD:\, R \subset Q} \int_R a\ci R(x,y)b\ci Q(y)dy \biggr|
\\
& =: A(x) +B(x)\,.
\end{align*}
Hence, using again the fact that  one of the two terms should at least a half of the sum, we can estimate
\begin{equation*}
|\{x: |\sha b(x)|>\lambda/2\} | \le |\{x: A(x)>\lambda/2\}|+ |\{x: B(x)>0\}|\,.
\end{equation*}
The second set is obviously inside $\cup_{ Q\in\cQ} Q$: indeed the function $B(x)$ vanishes outside this set because $a\ci R(x,y)=0$ for all $x\notin R$, and $R\subset Q$.  
So, using the property \cond3 of the \cz decomposition, we can estimate the measure of the second set as   
\[
|\{x: B(x)>0\}| \le 
\sum_{Q\in\mathcal{Q}}|Q|\le \frac{1}{\lambda}\|f\|_1. 
\]

To estimate the first measure we want to show that $\|A\|_1\le C\|f\|_1$, then clearly
\begin{equation}
\label{meas_A>la}
 |\{x: A(x)>\lambda/2\}| \le\frac 2\la \|A\|_1 \le \frac{2C}{\la} \|f\|_1. 
\end{equation}
We will estimate the norm of each term in $A$ separately. Let us
fix $Q\in\cQ$ and let us consider 
\[
A\ci Q (x) :=\sum_{R\in \cD, Q \subsetneqq R} \left| \int_R a\ci R(x,y) b\ci Q(y)dy \right| .
\] 
Since the function $b\ci Q$ is orthogonal to constants, and the function $a\ci R(x, \fdot)$ is constant on cubes $Q\in\cD$, $\ell(Q) <2^{-m}\ell(R)$, we can see that the only cubes $R$ which may contribute to $A_Q$ are the ancestors of $Q$ of orders $1,\ldots,m$. So, in general, there are at most $m$ non-zero terms in $A_Q$; if $\sha$ has scales separated by $r\geq m$ levels, there is at most one.

Recalling that for an integral operator $T$ with kernel $K$
\[
\|T\|\ci{L^1\to L^1} = \esssup_{y} \|K(\fdot, y)\|_1, 
\]
we can see that the integral operator with kernel $a\ci R$ is a contraction in $L^1$. Since at most $m$ such operators contribute to $A\ci Q$,
\[
\|A\ci Q\|_1 \le m \|b\ci Q\|_1 \le  2m\|\1_Q f\|_1;
\]
the last inequality here holds because of property \cond2 of the \cz decomposition.

Summing over all $Q\in\cQ$ we get 
\begin{align*}
\| A\|_1 \le  2m \sum_{Q\in\cQ} \|\1_Q f\|_1 \le 2m\|f\|_1.
\end{align*}
so (see \eqref{meas_A>la})
\[
|\{x: A(x)>\lambda/2\}| \le\frac 2\la \|A\|_1 \le \frac{4m}{\la} \|f\|_1.
\]
 If $\sha$ has scales separated by $r\geq m$ levels, we can take $1$ in place of $m$ in the last few estimates.
\end{proof}

Using this improved weak type estimate one can get the desired estimate \eqref{A2-test-1} by following the proof in \cite{LPR} and keeping track of the constants. 
However, there are several other places in \cite{LPR}, where the curse of exponentiality appears.
So for the convenience of the reader, we are doing all necessary estimates  below. Note that an analogous modification of \cite{LPR} was already carried out in \cite{H}; here we present yet another argument  in the spirit \cite{LPR} but with  modifications pertinent to eliminating the curse of exponentiality.

\subsection{First slicings} 

Let us fix $Q_0\in\cD$, and let us prove estimate \eqref{A2-test-1} for $Q=Q_0$. Recall, that $\sha$ is an integral operator with kernel $\sum_{Q\in\cD} a\ci{Q} (x, y)$, where $a\ci{Q}$ as in the previous section ($|Q|^{-1}$ is incorporated in $a\ci Q$). 

Define 
\begin{align*}
f\ci Q (x)  := \int_{Q_0} a\ci Q (x, y) w(y) dy, 
\end{align*}
so 
\begin{align*}
\sha (\1\ci{Q_0} w) = \sum_{Q\in\cD:\, Q\cap Q_0\ne \varnothing} f\ci Q =: f
\end{align*}
We can split $f$ into ``inner'' and ``outer'' parts, 
\begin{align*}
f = \sum_{Q\in\cD:\, Q\subset Q_0} f\ci Q + \sum_{Q\in\cD:\, Q_o\subsetneqq Q} f\ci Q =: f\ti i + f\ti o
\end{align*}

The ``outer'' part $f\ti o$ is easy to estimate. Since $\| a\ci Q(x, \fdot) \|_\infty \le |Q|^{-1}$, we can write for $Q_0\subsetneqq Q$
\begin{align*}
| f\ci Q (x)|\le w(Q_0) |Q|^{-1}
\end{align*}
and summing over all $Q$, $Q_0\subsetneqq Q$
\begin{align*}
|f\ti o (x) | \le |Q_0|^{-1} w(Q_0) \sum_{Q\in\cD :\, Q_0\subsetneqq Q} |Q|^{-1} |Q_0| =
|Q_0|^{-1} w(Q_0) \sum_{k=1}^\infty 2^{-kd} \le |Q_0|^{-1} w(Q_0). 
\end{align*}
Therefore,
\begin{align*}
\int_{Q_0} |f\ti o|^2 w^{-1} \le |Q_0|^{-2} w(Q_0)^2 w^{-1}(Q_0) \le [w]\ci{A_2} w(Q_0),  
\end{align*}


so $ \| \1\ci{Q_0} f\ti o \|\ci{L^2(w^{-1})}  \le [w]\ci{A_2}^{1/2} w(Q_0)^{1/2}$, and it only remains to estimate 
$\| f\ti i\|_{L^2(w^{-1})}$. 

Now we perform the first splitting. Let $r$ be the complexity of the shift $\sha$. Let us split the lattice $\cD$ into $r+1$ lattices $\cD_r^j$, $j=0, 1, \ldots, r$, where each lattice $\cD_r^j$ consists  of the cubes $Q\in \cD$ of size $2^{j - (r+1)\tau}$, $\tau\in \Z$.

If we can show that uniformly in $j$
\begin{align}
\label{eq:sum_D_r}
\int_{Q_0} \Bigl| \sum_{Q\in\cD^j_r} f\ci Q \Bigr|^2 w^{-1} \le C 2^{2d}   \left( B_2^2 + 1 \right)^2 [w]\ci{A_2}^2 w(Q_0),     
\end{align}
where $C$ is an absolute constant,  then we are done.   Indeed taking the sum over all $j=0, 1, \ldots , r$ we only multiply the estimate of the norm by $r+1$, so to get from the estimate \eqref{eq:sum_D_r}  to the desired estimate \eqref{A2-test-1} we just need to multiply the right side of \eqref{eq:sum_D_r} by $(r+1)^2$. 


The main reason for the this splitting of $\cD$ is that it simplifies the structure meaning that for $Q\in\cD^j_r$ the function $f\ci Q$ is constant on the children of $Q$ in the lattice $\cD^j_r$.  Also note that the shift $\sha^j f(x):=\sum_{Q\in\cD^j_r}\int_Q a_Q(x,y)f(y)dy$ has scales separated by $r+1>m$ levels, and $\1_{\cD^j_r}(Q)\cdot f_Q=\sha^j(\1_{Q_0}w)$.

Let us fix $j$, and let us from now on consider the lattice $\cD_r:= \cD_r^j$.  Since $j$ is not important in what follows, we will skip it and use the notation $\cD_r$, freeing the symbol $j$ for use in a different context. We also denote $\sha^j$ simply by $\sha$, bearing in mind the separation of scales which allows the use of the sharper estimate in the weak-type bound of Theorem~\ref{LPRnW}.

Now we split the lattice $\cD_r$ into the collections $\cQ_k$, $k\in\Z_+$, $ k < \log_2 ([w]\ci{A_2})$, where each $\cQ_k$ is the set of all cubes $Q\in \cD_r$ such that 
\begin{align}
\label{eq:Q_k}
2^{k}\le \frac{w(Q)}{|Q|} \cdot \frac{w^{-1}(Q)}{|Q|} < 2^{k+1}
\end{align}

We want to show that 
\begin{align}
\label{eq:sum_Q_k}
\int_{Q_0} \Bigl| \sum_{Q\in \cQ_k:\, Q\subset Q_0} f\ci Q \Bigr|^2 w^{-1} \le C_1 2^k [w]\ci{A_2}\, w(Q_0), 
\end{align}
where $C_1 = C 2^{2d}   \left( B_2^2 + 1 \right)^2$ is the constant in the right side of \eqref{eq:sum_D_r}. Then, using triangle inequality and summing the geometric progression we get 
\begin{align*}
\biggl\| 1\ci{Q_0} \sum_{Q\in\cD_r} f\ci Q\biggr\|_{L^2(w^{-1})} \le C_1^{1/2} [w]\ci{A_2}^{1/2} \ \sum_{k\in \Z_+:\,  k < \log_2([w]\ci{A_2})} 2^{k/2} \, w(Q_0)< 4 C_1^{1/2} [w]\ci{A_2}\, w(Q_0), 
\end{align*}
so \eqref{eq:sum_Q_k} implies that \eqref{eq:sum_D_r} holds with $C = 16 C_1$. 


So, we reduced the main result to the estimate \eqref{eq:sum_Q_k} with  $C_1  = C 2^{2d}   \left( B_2^2 + 1 \right)^2 $.   Note, that if we prove \eqref{eq:sum_Q_k} for $Q_0\in \cQ_k$, then we are done, because for general $Q_0$ we can add up the estimate for maximal subcubes of $Q_0$ belonging to $\cQ_k$.

\subsection{Stopping moments and Corona decomposition}

Let us suppose that the weight $w$ and the lattices $\cD_r$ and $\cQ=\cQ_k\subset \cD_r$  described above are fixed. 

Given a cube $Q_0\in\cQ=\cQ_k$ let us construct the generations 
 $\cG_\tau^*=\cG_\tau^*(Q_0)=\cG_\tau^*(Q_0, w, \cQ)$, $\tau\in \Z_+$ of stopping cubes as follows. 
Define the initial generation $\cG_0^*$ to be the cube $Q_0$. 

For all cubes  $Q\in \cG_\tau^*$   we consider maximal cubes $Q'\in\cQ$, $Q'\subset Q$ such that
\[
\frac{w(Q')}{|Q'|} > 4 \frac{w(Q)}{|Q|};
\]
the collection of all such cubes $Q'$ is the next generation $\cG_{\tau+1}^*$ of the stopping cubes. 

Let $\cG^*=\cG^*(Q_0):= \cup_{\tau\ge 0}\cG_\tau^* $ be the collection of all stopping cubes. 

Note, that if we start constructing stopping moments from a cube $Q\in\cG^*$, the stopping moments $\cG^*(Q)$ will agree with $\cG^*$, meaning that 
\begin{align*}
\cG^*(Q) =\{ Q'\in \cG^*: Q'\subset Q\}. 
\end{align*}

Let us introduce the last piece of notation. For a cube $Q\in\cG^*$  let us define $\cQ(Q) := \{Q'\in \cQ: Q'\subset Q\}$, and let 
\begin{align*}
\cP(Q):= \cQ(Q) \setminus \bigcup_{Q'\in \cG^*:\, Q'\subsetneqq Q} \cQ(Q'). 
\end{align*}
The above definitions make sense for arbitrary $Q\in\cQ$, but we will use it only for $Q\in\cG^*$, so we included this assumption in the definition.  Note that for $Q_0\in \cQ$ the set $\cQ(Q_0)$ admits  the following disjoint  decomposition
\begin{align}
    \label{dec_QQ_0}
    \cQ(Q_0)  = \bigcup_{Q\in \cQ^*(Q_0)} \cP(Q)
\end{align}

\subsubsection{Properties of stopping moments}
It follows from the construction of $\cG^*$ that  if $R\in\cG^*$ and $Q$ is a maximal cube in $\cG^*$ such that $Q\subsetneqq R$, then  
\begin{equation}
\label{up4}
\frac{w(Q)}{|Q|} > 4 \frac{w(R)}{|R|}\,.
\end{equation}
The estimate \eqref{up4} implies 
\begin{equation}
|Q| \le \frac{|R|}{4}\cdot \frac{w(Q)}{w(R)}, 
\end{equation}
and summing over all such maximal $Q\in\cG^*$, $Q\subsetneqq R$ (assume that $R\in\cG_\tau^*$) we get
\begin{equation}
\label{eq:Carl-1}
\Bigl| \bigcup_{Q\in\cG^*:\,Q\subsetneqq R} Q \Bigr| =\sum_{Q\in \cG_{\tau+1}^*: \,Q\subsetneqq R} |Q| \le \frac{|R|}{4 w(R)} \sum_{Q\in \cG_{\tau+1}^*: \,Q\subsetneqq R} w(Q) \le \frac14 |R|,  
\end{equation}
for all $R\in \cG^*$.

Repeating this estimate for each  $Q$ and summing over the generations we get
\[
\sum_{Q\in \cG^*: Q\subsetneqq R} |Q| \le |R| \sum_{n=1}^\infty 4^{-n} = \frac13 |R|.
\]
Adding $|R| $ to this sum we get that the following \emph{Carleson property} of the stopping moments $\cG^*$ 
\begin{equation}
\label{carlL}
\sum_{Q\in \cG^*: Q\subset R} |Q| \le   \frac43 |R| .
\end{equation}
It is easy to see that this estimate holds for all $R\in\cD$, not just for $R\in\cG^*$: one just needs to consider maximal cubes $R'\in\cG^*$, $R'\subset R$ and apply \eqref{carlL} to each of  these cubes. 

Iterating \eqref{eq:Carl-1} and summing over all generations we get 
\begin{equation}
\label{qadrL}
\biggl\|\sum_{Q\in\cG^*, Q\subset R} \1\ci{Q} \biggr\|_2 \le  |R|^{1/2} \sum_{k=0}^\infty 2^{-k} = 2|R|^{1/2}\,.
\end{equation}

We need the following simple lemma      
\begin{lm}
\label{35}
For any $R\in \cD$
\begin{equation}
\label{carlLW}
\sum_{Q\in\cG^*, Q\subset R} w(Q) \le C[w]_{A_2} w(R)\,, 
\end{equation}
where $C$ is an absolute constant.   
\end{lm}
\begin{proof}
The Carleson Embedding Theorem (see Theorem \ref{t:Carl} above) applied to $\1\ci R$ together with the Carleson property \eqref{carlL} imply that 
\[
\sum_{Q\in\cG^*, Q\subset R} \left( {\fint_Q} w^{1/2} \right)^2 |Q| \le C \|\1\ci R w^{1/2} \|_2^2 = C w(R). 
\]
(the best constant is $C=4\cdot 4/3$). But 
\begin{align*}
\left( \fint w^{1/2} \right)^{-1} & \le \fint w^{-1/2} \le \left( \fint w^{-1} \right)^{1/2}  
&& \text{by Cauchy--Schwartz}
\\
& \le [w]\ci{A_2}^{1/2} \left( \fint_Q w  \right)^{-1/2}    &&
\text{because } \left( \fint_Q w  \right) \left( \fint_Q w^{-1}  \right) \le [w]\ci{A_2}, 
\intertext{so} 
 & \fint_Q w \le [w]_{A_2}\left( \fint_Q w^{1/2}  \right)^{2}  &&
\end{align*}
and the lemma is proved (with $C=16/3$).
\end{proof}

This proof was  (essentially) present in \cite{Wav_PastFuture}. In \cite{LPR} a different proof, using a clever  iteration argument   and giving  the better constant $C= 16/9$, was presented. 

\subsection{John--Nirenberg type estimates}
\label{s:J-N}

Given a collection $\cA$ of cubes, $\cA\subset \cD_r$, define the 
function $f\ci{\!\!\cA}$ by  
\begin{align*}
f\ci{\!\!\cA} := \sum_{Q\in \cA} f\ci Q. 
\end{align*}
For the cube cube $Q_0\in \cG^*$ consider the function $f\ci{\!\!\cQ(Q_0)}$. By \eqref{dec_QQ_0} the function $f\ci{\!\!\cQ(Q_0)}$ can be decomposed as 
\begin{equation}
\label{dec-f_QQ_0}
f\ci{\!\!\cQ(Q_0)}=\sum_{R\in\cG^*} f\ci{\!\!\cP(R)}\,,
\end{equation}
where recall $\cG^*:=\cG^*(Q_0)$ is the collection of stopping cubes. 

The main reason for introducing this decomposition is that, as we will show below, the functions $f\ci{\cP(R)}$ behave in many respects as BMO functions: they have exponentially decaying distribution functions, so, in particular all $L^p$ norms for $p<\infty$ are equivalent. 

In the  proof of these facts the weak $L^1$ estimate of dyadic shifts (Theorem \ref{LPRnW}) is used. 

The first lemma, which is Lemma 3.15 in \cite{LPR}, is a simple observation, that for the John--Nirenberg estimates of the distribution function it is sufficient to have weak type estimates.

Recall that $\cD_r$ is $2^{r}$-adic lattice, i.e.~the children $Q'$ of $Q$ satisfy $\ell(Q') =2^{-r} \ell(Q)$. 

\begin{df}
\label{df:maxf}
Let $\phi\ci Q$, $Q\in \cD_r$ be a collection of functions such that $\phi\ci Q$ is supported on $Q$ and is constant on children (in $\cD_r$) of $Q$. For $R_0\in\cD_r$ let   $\phi^*\ci{R_0}$ be a maximal function
\[
\phi^*\ci{R_0} (x) := \sup_{Q\in\cD_r: Q\ni x} \ \Bigl| \sum_{R\in\cD_r :\, Q\subsetneqq R \subset R_0} \phi\ci R(x) \Bigr|. 
\] 
\end{df}

\begin{lm}
\label{lm:John-Nir-1}
Let $\phi\ci Q$, $Q\in\cD_r$ be a collection of functions such that 
\begin{enumerate}
    \item $\phi\ci Q$ is supported on $Q$ and constant on the children (in $\cD_r$) of $Q$;
    \item $\|\phi\ci Q \|_\infty\le 1$; 
    \item There exists $\delta\in(0,1)$ such that for all cubes $R\in \cD_r$ 
    \[
    \bigl| \Bigl\{  x\in R :   \phi^*\ci R(x)  > 1 \Bigr\} \bigr| \le \delta |R| \,.
    \]
\end{enumerate}
Then for all $R\in\cD_r$ and for all $t\geq 0$
\[
\bigl| \Bigl\{  x\in R :   \phi^*\ci R(x)  > t \Bigr\} \bigr| \le
\delta^{(t-1)/2} |R|\,.
\]
\end{lm} 
\begin{proof}
Let us prove the conclusion of the lemma for a fixed cube $R=R_0\in\cD_r$. 

Let $\cB_1$ be the collection of all maximal cubes $Q\in\cD_r$, $Q\subset R_0$ such that 
\begin{align}
\label{eq:J-N_stop-1}
\Bigl| \sum_{R\in\cD_r :\, Q\subsetneqq R \subset R_0} \phi\ci R(x) \Bigr|  > 1, \qquad x\in Q; 
\end{align}
note that the functions $\phi\ci R$ (and so the sum) are constant on the cube $Q$. 

Define the set $B_1$, 
\[
B_1:= \bigcup_{Q\in \cB_1} Q. 
\]
It follows from the construction that $\phi^*\ci{R_0}\le 1$ outside of $B_1$, and that for any $Q\in \cB_1$ the sum in \eqref{eq:J-N_stop-1} is at most $2$. Note also that by the assumption \cond3 we have that $|B_1| \le \delta |R_0|$. 

For each cube $\wt R\in\cB_1$ we repeat the above construction (with $\wt R$ instead of $R_0$); we will get a collection of stopping cubes $\cB_2$ and the set $B_2= \cup_{Q\in\cB_2} Q$, $B_2\subset B_1$, $|B_2| \le \delta^2 |R_0|$. It is easy to see that $\phi\ci{R_0}^* \le 2+1=3$ outside of $B_2$ and that for any cube $Q\in \cB_2$
\begin{align*}
\Bigl| \sum_{R\in\cD_r :\, Q\subsetneqq R \subset R_0} \phi\ci R(x) \Bigr| \le 4, \qquad x\in Q
\end{align*}
(sums outside of $\wt R \in\cB_1$ contribute at most $2$, and the sums starting at $\wt R\in\cB_1$ contribute at most $1$ outside of $B_2$ and at most $2$ on $Q\in \cB_2$. 

Repeating this procedure  we get the collections $\cB_n$ of ``stopping cubes'' and the decreasing sequence of sets $B_n=\cup_{Q\in\cB_n} Q$, such that 
\begin{align}
\label{|B_n|}
 |B_n| & \le \delta^n;  &&\\    
\label{phi_le_2n-1}
\phi^*\ci{R_0} & \le 2n-1  &&\text{outside of } B_n; \\
\notag
 \Bigl| \sum_{R\in\cD_r :\, Q\subsetneqq R \subset R_0} \phi\ci R(x) \Bigr| & \le 2n  &&\forall Q\in \cB_n, \ \forall x \in Q;
\end{align}
the last inequality is only needed for the inductive construction. 

Given $t>1$ let $n$ be the largest integer such that $2n-1\le t$, 
\[
n =\left\lfloor {(t+1)}/{2} \right\rfloor . 
\] 
By \eqref{phi_le_2n-1}
\[
\phi^*\ci{R_0} \le 2n-1 \le t\qquad \forall x\notin B_n, 
\]
so
\[
\bigl| \Bigl\{x\in R_0: \phi^*\ci{R_0}(x) > t \Bigr\} \bigr| \le |B_n| \le \delta^n \le \delta^{(t-1)/2} .
\]
This completes the proof for $t>1$, but for $0\leq t\leq 1$ the conclusion is trivial.
\end{proof}

As it was shown above in Theorem \ref{LPRnW}, the weak $L^1$ norm of a dyadic shift $\sha$ of complexity $r$, with scales separated by $r+1$ levels, can be estimated by $C=2^{d+2}\|\sha\|_2^2 + 5$, so the weak $L^1$ norm of our dyadic shift $\sha$ and all its subshifts $\sha\ci \cA$, $\cA\subset \cD_r$, can be estimated by
\begin{equation}
\label{eq:B_1}
B_1 = 2^{d+2} B_2^2 + 5, 
\end{equation}
where 
\[
B_2= \sup_{\cA\subset\cD} \|\sha\ci \cA\|\ci{L^2\to L^2} . 
\]

Now we need the following lemma, which is essentially Lemma 4.7 from \cite{LPR} with all constant written down; in fact, certain modifications in the argument are needed to avoid introducing exponential dependence on $r$, which was (implicitly) the case in \cite{LPR}. Such a modification (with linear dependence on $r$) was first obtained in Lemma~7.2 of \cite{H}; here we even achieve an estimate uniform with respect to $r$ by taking into account the separation of scales of our shift, and the resulting improvement in the estimate of Theorem~\ref{LPRnW}.

Let $\cP\subset\cD_r$ be a collection of cubes. Define the maximal function $f^*\ci{\!\!\cP}$ (compare with Definition \ref{df:maxf}) by 
\begin{equation}
\label{eq:f*PR}
f\ci{\!\!\cP}^* (x) := 
 \sup_{Q\in\cD_r: Q\ni x} \ \Bigl| \sum_{R\in\cP :\, Q\subsetneqq R} f\ci R(x) \Bigr|. 
\end{equation}
For the function $f\ci{\!\!\cP(R_0)}$, $R_0\in\cG^*$ defined above in the beginning of Section \ref{s:J-N} we have $|f\ci{\!\!\cP(R_0)}| \le f^*\ci{\!\!\cP(R_0)}$, so we will use $f^*\ci{\!\!\cP(R_0)}$ to estimate the distribution function of $|f\ci{\!\!\cP(R_0)}|$.

Note that for $R\in \cP$ we cannot guarantee that its children 
in $\cD_r$ are in $\cP$. So while in the above definition the 
sums are taken over all $R\in \cP$, we need to 
take supremum over $Q\in\cD_r$. 
\begin{lm}
\label{47}
Let $B_1$ is given by \eqref{eq:B_1}. Then for 
any $R\in \cG^*$ we have
\begin{align}
\label{48}
\Bigl| \biggl\{ x\in R: f^*\ci{\!\cP(R)} (x)> 16 t\frac{w(R)}{|R|} \biggr\}\Bigr|  & \le 2\sqrt2 \cdot 2^{-t/2B_1} |R|\,, \\
\label{49}
w^{-1} \Bigl( \biggl\{ x\in R: f^*\ci{\!\cP(R)} (x)> 20 t\frac{w(R)}{|R|} \biggr\} \Bigr)   & \le  24\cdot 2^{-t/2B_1} w^{-1}(R)\,,
\end{align}
\end{lm}
\begin{proof}
Now it is time to perform the last splitting. Namely, let us split the set $\cP(R)$ into the sets $\cP_\alpha(R)$, 
$\alpha\in \Z_+$, where the collection $\cP_\alpha = \cP_\alpha(R)$ consists of all cubes $Q\in \cP(R)$ for which
\begin{align}
\label{P_alpha}
4^{-\alpha } \frac{w(R)}{|R|} < \frac{w(Q)}{|Q|} \le 4^{-\alpha + 1} \frac{w(R)}{|R|}. 
\end{align}
Note, that by the construction of stopping moments
\[
\frac{w(Q)}{|Q|} \le 4 \frac{w(R)}{|R|}
\]
so we do not need $\alpha<0$. 

We can estimate
\[
f^*\ci{\! \cP(R)} \le \sum_{\alpha\in\Z_+} f^*\ci{\! \cP_\alpha (R)}
\]

Now let us estimate the level sets of $f^*\ci{\!\cP_\alpha(R)}$ using the above Lemma \ref{lm:John-Nir-1}. For $Q\in\cP_\alpha(R)$
\[
|f\ci Q(x)| \le \frac{w(Q)}{|Q|} \le 4^{-\alpha + 1} \frac{w(R)}{|R|} \le \frac{w(R)}{|R|} 2^{-2\alpha + 3}  B_1
\]
(recall that $B_1\ge 1$).

To this end, let
\begin{equation*}
   \phi_Q:=1_{\cP_{\al}(R)}(Q)\cdot\frac{2^{2\alpha-3}|R|}{B_1w(R)}\cdot f_Q,
\end{equation*}
so that $\phi_Q$ satisfies the first two assumptions of Lemma~\ref{lm:John-Nir-1}.

Recall the notation $\phi_{R_1}^*$ from Definition~\ref{df:maxf}. We want to use the weak type estimate for shifts to estimate the size of the set
\begin{equation*}
  \Bigl\{ x\in R_1 :  \phi^*_{R_1} > 1\Bigr\}.
\end{equation*}
Observe that this set is the union of the maximal cubes $M\in\mathscr{D}_r$ such that
\begin{equation*}
  \Bigl|\sum_{M:M\subsetneqq Q\subset R_1}\phi_Q(x)\Bigr|>1
\end{equation*}
for $x\in M$. Let $\mathscr{M}$ stand for the collection of these maximal cubes, and let
\begin{equation*}
  \mathscr{N}:=\{Q\in\mathscr{D}_r: Q\subset R_1;\not\exists M\in\mathscr{M}, Q\subset M\}.
\end{equation*}
Then
\begin{equation*}
  \Bigl\{ x\in R_1 :  \phi_{R_1}^*>1\Bigr\}
  =\Bigl\{ x\in R_1 :  \Bigl|\sum_{Q\in\mathscr{N}}\phi_Q\Bigr| > 1\Bigr\},
\end{equation*}
where
\begin{equation*}
  \sum_{Q\in\mathscr{N}}\phi_Q
  =\frac{2^{2\alpha-3}|R|}{B_1w(R)}\sum_{Q\in\cP_{\al}(R)\cap\mathscr{N}} f_Q
  =\frac{2^{2\alpha-3}|R|}{B_1w(R)}\sha_{\cP_{\al}(R)\cap\mathscr{N}}(\1_{R_1}w);
\end{equation*}
hence, by Theorem~\ref{LPRnW} and  $\|\1_{R_1}w\|_1=w(R_1)$,
\begin{equation*}
  \Bigl|\Bigl\{ x\in R_1 :  \phi_{R_1}^*>1\Bigr\}\Bigr|
  \leq B_1\frac{2^{2\alpha-3}|R|}{B_1 w(R)}w(R_1).
\end{equation*}
If $R_1\in\cP_{\al}(R)$, then the right side is directly dominated by $2^{2\alpha-3}\cdot 4^{-\alpha+1}|R_1|=\tfrac12|R_1|$. For an arbitrary $R_1\in\cD_r$, observe that $\phi_{R_1}^*=\sum_P\phi_P^*$, where the summation ranges over the maximal $P\in\cP_{\al}(R)$ with $P\subset R_1$. Since $\supp\phi_P^*\subset P$, and these cubes are disjoint, it follows that
\begin{equation*}
  \Bigl|\Bigl\{ x\in R_1 :  \phi_{R_1}^*>1\Bigr\}\Bigr|
  =\sum_P\Bigl|\Bigl\{ x\in P :  \phi_P^*>1\Bigr\}\Bigr|\leq\sum_P\frac12|P|\leq\frac12|R_1|.
\end{equation*}
Observing that
\begin{equation*}
  \phi_R^*=\frac{2^{2\alpha-3}|R|}{B_1w(R)}\cdot f_{\cP_{\al}(R)}^*,
\end{equation*}
Lemma~\ref{lm:John-Nir-1} implies that
\begin{equation*}
  \Bigl|\Bigl\{ x\in R :  f_{\cP_{\al}(R)}^*>t\frac{w(R)}{|R|}2^{-2\alpha+3} t\Bigr\}\Bigr|
  =\Bigl|\Bigl\{ x\in R :  \phi_R^*>t\Bigr\}\Bigr|\leq 2^{-(t-1)/2}|R|.
\end{equation*}


%
Rescaling  $t$ we can rewrite the inequality as  
\begin{equation}
\label{|E_alpha|}
\bigl| \Bigl\{ x\in R :  f^*\ci{\!\cP_\alpha(R)}(x) > 16 t \frac{w(R)}{|R|} 
\Bigr\} \bigl| \le \sqrt2 \cdot 2^{- t 4^{\alpha} /B_1} |R|  \qquad \forall t>0. 
\end{equation}
Denote the set above as $E_\alpha(t)$, 
\[
E_\alpha(t) :=\Bigl\{ x\in R :  f^*\ci{\!\cP_\alpha(R)}(x) > 16 t \frac{w(R)}{|R|} 
\Bigr\}. 
\]

We want to estimate the set where 
\[
\sum_{\alpha=0}^\infty f^*\ci{\!\!\cP_\alpha(R)}(x) >T .
\]
If this happens for $x\in R$, then either $f^*\ci{\!\!\cP_0(R)}(x) >T/2$, or 
\[
\sum_{\alpha=1}^\infty f^*\ci{\!\!\cP_\alpha(R)}(x) >T/2 .
\]
The latter inequality implies that either $f^*\ci{\!\!\cP_0(R)}(x) >T/4$ or 
\[
\sum_{\alpha=1}^\infty f^*\ci{\!\!\cP_\alpha(R)}(x) >T/4 , 
\]
and so on. 

Repeating this reasoning with $T= 16 w(R) t/|R|$, we can see that 
\begin{align*}
\Bigl\{ x\in R :  f^*\ci{\!\cP (R)}(x) > 16 t \frac{w(R)}{|R|} 
\Bigr\} \subset \bigcup_{\alpha\ge0} E_\alpha(2^{-\alpha-1} t)
\end{align*}
so using \eqref{|E_alpha|} we get
\begin{align*}
|R|^{-1} \bigl|  \Bigl\{ x\in R :  f^*\ci{\!\cP(R)}(x) > 16 t \frac{w(R)}{|R|} 
\Bigr\} \bigr| 
&\le \sqrt 2 \sum_{\alpha=0}^\infty 2^{- t \cdot 2^{\alpha-1}/B_1} &&
 \\
&\le \sqrt2 \sum_{\alpha=0}^\infty 2^{- t/2B_1 -\alpha}  && \text{if } t\ge 2B_1 \\
& \le 2\sqrt2\cdot 2^{- t/2B_1}
\end{align*}
which proves \eqref{48}. We have proved \eqref{48} for $t\ge 2B_1$, but for $t<2B_1$ this estimate is trivial, because the right side is greater than $|R|$. Thus, \eqref{48} holds for all $t>0$. 

To prove \eqref{49}, let us first recall that all our cubes are in $\cQ=\cQ_k$, so  \eqref{eq:Q_k} holds for all of them. If, in addition $Q\in \cP_{\alpha}(R)$, then  \eqref{P_alpha} (the definition of $\cP_\alpha(R)$) is satisfied, and combining these two estimates we get 
\begin{align}
    \label{eq:P_alpha-2}
    2^k 4^{\alpha-1} \frac{|R|}{w(R)}  \le  \frac{w^{-1}(Q)}{|Q|}   \le  
    2^{k+1}4^{\alpha} \frac{|R|}{w(R)}  \qquad \forall Q\in \cP_\alpha(R). 
\end{align}
So $w^{-1}(Q)$ can be estimated via $|Q|$, so we will use the known estimates of the Lebesgue measure of level sets to get the estimates of the $w^{-1}$ measure. 

Let us consider the set  where 
\begin{align*}
f^*\ci{\!\!\cP_\alpha(R)}(x) > 20 t \frac{w(R)}{|R|} .
\end{align*}
This set is a disjoint union of cubes $Q'\in\cD_r$, which are the first (maximal) cubes $Q$ for which the sum  in \eqref{eq:f*PR} defining $f^*\ci{\!\!\cP_\alpha(R)}$  
exceeds $20t \cdot w(R)/|R|$. Unfortunately the cubes $Q'$ are not necessarily in $\cP_\alpha(R)$, so we cannot use \eqref{eq:P_alpha-2} for them. But their parents are in $\cP_\alpha(R)$ (because the summation is over $\cP_\alpha(R)$)!

So, let $\cE_\alpha(t)$ be the collection of such parents, and let  
\[
\wt E_\alpha(t):= \bigcup_{Q\in \cE_\alpha(t)}  Q.
\]
 Note, that to get $\wt E_\alpha(t)$ it is sufficient  to take the union of the maximal cubes $Q\in \cE_\alpha(t)$, so the set $\wt E_\alpha(t)$ is a disjoint union of cubes $Q\subset \cP_\alpha(R)$. Since for $Q \in\cP_\alpha(R)$
\[
|f\ci Q(x) | \le \frac{w(Q)}{|Q|} \le 4^{-\alpha+1} \frac{w(R)}{|R|} ,
\]
we can conclude that for all  $Q\in \cE_\alpha(t)$ and all $t\ge 4^{-\alpha}$
\begin{align*}
\Bigl|  \sum_{R'\in \cP_\alpha(R): \, Q\subsetneqq R'}  f\ci{R'}(x) \Bigr| \ge 20 t \frac{w(R)}{|R|} - 4 \cdot 4^{-\alpha} \frac{w(R)}{|R|} \ge 16t \frac{w(R)}{|R|} \qquad \forall x\in Q 
\end{align*}
(because the corresponding sum for one of the children $Q'$ of $Q$ exceeds $20 t \cdot w(R)/|R|$ on $Q'$, and the difference between the two sums is $f\ci Q$; we also use that the sum in the left hand side is constant on $Q$). 

So $f^*\ci{\!\!\cP_\alpha(R)} (x) > 16 t \cdot w(R)/|R|$ on $Q$, and we conclude that for $t\ge 4^{-\alpha}$ the inclusion $\wt E_\alpha(t)\subset E_\alpha(t)$ holds. Using the estimate \eqref{|E_alpha|} for $|E_\alpha(t)|$ (and replacing $\sqrt2$ by $2$ there) we get that for $t\ge 4^{-\alpha}$
\begin{align}
\label{|wtE_alpha|}
|\wt E_\alpha(t) | \le 2  \cdot 2^{- t 4^{\alpha} /B_1} |R| . 
\end{align}
Note that for $t<4^{-\alpha}$ the above estimate is trivial, so it holds for all $t>0$.  

Since by \eqref{eq:P_alpha-2} for all $Q\in \cP_\alpha(R)$
\[
w^{-1}(Q) \le  2^{k+1} 4^{\alpha} \frac{|R|}{w(R)} |Q|
\]
summing over maximal cubes in $\cE_\alpha(t)$ we get 
\begin{align}
\notag
w^{-1}(\wt E_\alpha(t) ) & \le 2^{k+1} 4^{\alpha} \frac{|R|}{w(R)} |\wt E_\alpha(t)|  &&
\\
\notag
&\le   2^{k+1} 4^{\alpha} \frac{|R|}{w(R)} 2  \cdot 2^{- t 4^{\alpha} /B_1} |R|   && \text{by } \eqref{|wtE_alpha|} 
\\
\label{eq:wE_alpha}
& \le 4^{\alpha}\cdot 2^2\cdot   2^{- t 4^{\alpha} /B_1} w^{-1}(R)       &&  \text{by } \eqref{eq:Q_k}  
\end{align}

%
Now we want to estimate $w^{-1}(\wt E(t))$, where 
\[
\wt E(t) := \Bigl\{ x\in R :\, f^*\ci{\!\!\cP} (x) > 20 t \frac{w(R)}{|R|} \Bigr\}  .
\]
Let $T:= 20 t \cdot w(R)/R$. If for $x\in R$
\[
\sum_{\alpha=0}^\infty f^*\ci{\!\!\cP_\alpha(R)}(x) >T , 
\]
then either $f^*\ci{\!\!\cP_0(R)}(x) > T/2$ (in which case $x\in \wt E_0(t/2)$) or
\[
\sum_{\alpha=1}^\infty f^*\ci{\!\!\cP_\alpha(R)}(x) >T/2 .
\]
If the latter inequality holds, then either $f^*\ci{\!\!\cP_1(R)}(x) > T/4$, so $x\in \wt E_0(t/4)$, or 
\[
\sum_{\alpha=2}^\infty f^*\ci{\!\!\cP_\alpha(R)}(x) >T/4 .
\]
Repeating this reasoning we get that 
\[
\wt E(t) \subset \bigcup_{\alpha\ge 0} \wt E_\alpha(t2^{-\alpha-1}), 
\]
so 
\begin{align*}
w^{-1} (\wt E(t)) & \le \sum_{\alpha= 0}^\infty w^{-1} ( \wt E_\alpha(t2^{-\alpha-1})  )   &&
\\
& \le  4 w^{-1}(R)  \sum_{\alpha= 0}^\infty  4^{\alpha} 2^{- t 2^{\alpha - 1} /B_1}       &&  \text{by } \eqref{eq:wE_alpha}
\\
& \le 4 w^{-1}(R)\cdot 6\cdot 2^{-t/2B_1} && \text{if } t\ge 2B_1. 
\end{align*}
To prove the last inequality we need for $t\ge 2B_1$ to estimate the sum 
\[
\sum_{\alpha=0}^\infty 2^{2\alpha - t 2^\alpha /2B_1}. 
\]
Since $2^\alpha \ge 3\alpha +2$ for $\alpha \ge 4$, we can estimate for $\alpha\ge 4$ and $t\ge 2B_1$
\begin{align*}
2\alpha - t 2^\alpha/2B_1 & \le 2\alpha - t\cdot (3\alpha +2)/2B_1  
\\
& = \bigl(2\alpha - 2\alpha t /2B_1\bigr) - \alpha t /2B_1 - 2 t /2B_1 
\\
& \le 0  -\alpha -  t/2B_1, 
\end{align*}
so 
\[
\sum_{\alpha=4}^\infty 2^{2\alpha - t 2^\alpha /2B_1} \le 2^{-t/2B_1} \sum_{\alpha=4}^\infty 2^{-\alpha} < \frac12\cdot2^{-t/2B_1}  .
\]
For $\alpha=0, 1, 2, 3$ we can estimate
\[
2^{2\alpha - t 2^\alpha /2B_1} \le c_\alpha 2^{-t/2B_1}   , \qquad \text{where } \ c_0=1,\ c_1=c_2=2, \ c_3=\frac12,
\]
so adding everything we get that
\[
w^{-1} (\wt E(t)) \le 24 \cdot 2^{-t/2B} w^{-1}(R).  
\] 
We proved that estimate for $t\ge 2B_1$, but for $t<2B_1$ the estimate is trivial because the right side is bigger than $w^{-1}(R)$. So the estimate holds for all $t>0$. 
\end{proof}

\subsection{Conclusion of the proof}

\begin{lm}
For any $R\in \cG^*$
\begin{align}
\label{norm-f_P(R)-1}
\|f\ci{\!\!\cP(R)} \|\ci{L^2}  & \le C_1 B_1\frac{w(R)}{|R|} |R|^{1/2}, \\
\label{norm-f_P(R)-w}
\|f\ci{\!\!\cP(R)} \|\ci{L^2(w^{-1})} & \le C_1 B_1 \frac{w(R)}{|R|}  \sqrt{w^{-1}(R)}, 
\end{align}
where $C_1$ and $C_2$ are absolute constants and $B_1$ is given by \eqref{eq:B_1}
\end{lm}

This lemma is proved by using the distributional inequalities from Lemma \ref{47} and computing the norms using distribution functions. That will give the desired estimates for the norms of the maximal function $f^*\ci{\!\!\cP(R)}$, and since $| f\ci{\!\!\cP(R)}| \le f^*\ci{\!\!\cP(R)}$, we get the conclusion of the lemma. We leave the details as a trivial exercise for the reader.  

Recall, that to prove the main result we need to prove estimate \eqref{eq:sum_Q_k} for all cubes $Q_0 \in \cQ =\cQ_k$. For a cube $Q\in \cQ$, let $\cQ(Q) := \{Q'\in \cQ:\, Q'\subset Q\}$. We want to estimate $\|f\ci{\cQ(Q_0)}\|\ci{L^2(w^{-1})}$, $Q_0\in \cQ$, where 
\[
f\ci{\cQ(Q_0)} := \sum_{Q\in \cQ(Q_0)} f\ci Q. 
\]
Since (see \eqref{dec-f_QQ_0})
\[
f\ci{\cQ(Q_0)} = \sum_{Q\in \cG^*(Q_0)} f\ci{\!\!\cP(Q)} , 
\]
we can write 
\begin{align*}
\| f\ci{\cQ(Q_0)} \|\ci{L^2(w^{-1})}^2 & \le \sum_{R\in \cG^*(Q_0)} \| f\ci{\!\!\cP(R)} \|\ci{L^2(w^{-1})}^2 + 
2 \ \sum_{R, Q \in \cG^*(Q_0):\, Q\subsetneqq R} \bigl| \langle f\ci{\!\!\cP(R)} , f\ci{\!\!\cP(Q)} \rangle_{w^{-1}} \bigr|
\\
& = S_1 + S_2. 
\end{align*}
The first sum is easy to estimate. By \eqref{norm-f_P(R)-w}
\begin{align*}
\|f\ci{\!\!\cP(R)} \|\ci{L^2(w^{-1})}^2  & \le [C_1 B_1]^2 \frac{w(R)^2}{|R|^2} w^{-1}(R), &&
\\
& \le [C_1 B_1]^2  2^{k+1} w(R) .  && \text{because } R \in \cQ=\cQ_k
\end{align*}


Summing over all $R\in \cG^*= \cG^*(Q_0)$ we get  using \eqref{carlLW}
\begin{align*}
S_1 \le  2 [C_1 B_1]^2 2^k \sum_{R\in \cG^*(G_0)} w(R) \le C  B_1^2 2^k [w]\ci{A_2} w(Q_0), 
\end{align*}
where $C$ is an absolute constant.

Let us now estimate $S_2$.

Let $Q, R\in \cG^*$, $Q\subsetneqq R$. Then $f\ci{\!\!\cP(R)}(x)$ is constant on $Q$, let us use the symbol 
$f\ci{\!\!\cP(R)}(Q)$ to denote this constant. 
We then can estimate 
\begin{align}
\notag
\left| \langle f\ci{\!\!\cP(R)}, f\ci{\!\!\cP(Q)}\rangle_{w^{-1}} \right|  & \le | f\ci{\!\!\cP(R)}(Q) | \cdot (w^{-1}(Q))^{1/2} 
\| f\ci{\!\!\cP(Q)}\|\ci{L^2(w^{-1})} && \text{by Cauchy--Schwartz} 
\\
\notag
& \le C_1B_1 | f\ci{\!\!\cP(R)}(Q) | \frac{w^{-1}(Q) w(Q)}{|Q|}   && \text{by } \eqref{norm-f_P(R)-w}
\\
\label{fPR,fPQ}
& \le C_1B_1 | f\ci{\!\!\cP(R)}(Q) | 2^{k+1} \cdot |Q| && \text{because } Q\in \cQ_k.
\end{align}


Using this estimate we can write 
\begin{align*}
S_2 (R) & := \sum_{Q \in \cG^*:\, Q\subsetneqq R} \left| \langle f\ci{\!\!\cP(R)}, f\ci{\!\!\cP(Q)}\rangle_{w^{-1}} \right|  
&& 
\\
& \le 2^{k+1} C_1 B_1 \sum_{Q \in \cG^*:\, Q\subsetneqq R} |f\ci{\!\!\cP(R)}(Q)|  \cdot  |Q|  
&& \text{by } \eqref{fPR,fPQ}
\\
& = 2^{k+1} C_1 B_1 \int_R | f\ci{\!\!\cP(R)} | \sum_{Q \in \cG^*:\, Q\subsetneqq R} \1\ci Q \ dx  &&
\\
& \le 2^{k+1} C_1 B_1 \bigl\| f\ci{\!\!\cP(R)} \bigr\|_2 \cdot \Bigl\| \sum_{Q \in \cG^*:\, Q\subsetneqq R} \1\ci Q \Bigr\|_2 && \text{by Cauchy--Schwartz}
\\
& \le 2^{k+2} [C_1 B_1]^2 w(R)   && \text{by } \eqref{norm-f_P(R)-1} \text{ and } \eqref{qadrL}
\end{align*}
Therefore, using \eqref{carlLW}
\begin{align*}
S_2 \le 2^{k+1} [C_1 B_1]^2  & \le   2^{k+1} [C_1 B_1]^2  \sum_{R\in \cG^*(Q_0)} w(R) 
\\
& \le C (B_1)^2 2^k [w]\ci{A_2} w(Q_0) 
\end{align*}
with some absolute constant $C$. But that is exactly the estimate \eqref{eq:sum_Q_k}, so Theorem \ref{t:sharp-shift-A2} is proved.    \hfill \qed




\begin{thebibliography}{XXX}
\label{rf}

\bibitem{AIS} {\sc K. Astala, T. Ivanec, E. Saksman, } {\em Betrami operators in the plane}, Duke Math J., 107 (2001), 27-56.

\bibitem{Buck1} {\sc S. M. Buckley}, {\em Estimates for operator norms on weighted spaces 
and reverse Jensen inequalities},  Trans. Amer. Math. Soc., {\bf 340} (1993), no. 1, p53--272.



\bibitem{OB} {\sc O.  Beznosova}, {\em Linear bound for the dyadic paraproduct on weighted Lebesgue space $L^2(w)$}, J. Funct. Analysis, {\bf 255} (2008), No. 4, 994--1007.

\bibitem{CUMP1}{\sc D. Cruz-Uribe, J. Martell, C. Perez}, {\em Sharp weighted estimates for approximating dyadic operators}, accepted in Electronic Research Announcements in the Mathematical Sciences


\bibitem{CUMP2}{\sc D. Cruz-Uribe, J. Martell, C. Perez}, {\em Sharp weighted estimates for classical operators,} arXiv:1001.4724.


\bibitem{Dav} {\sc G. David,}  {\em Analytic capacity, Calder\'on-Zygmund operators, and rectifiability},
Publ. Mat., {\bf 43} (1999), 3--25.

\bibitem{Dr-Volb-A2_Beurling_2003}
{\sc Oliver Dragi{\v{c}}evi{\'c} and Alexander Volberg}, \emph{Sharp estimate of the
  {A}hlfors-{B}eurling operator via averaging martingale transforms}, Michigan
  Math. J. \textbf{51} (2003), no.~2, 415--435. 
  
  
\bibitem{Duren1970}
{\sc Peter~L. Duren}, \emph{Theory of {$H\sp{p}$} spaces}, Pure and Applied
  Mathematics, Vol. 38, Academic Press, New York, 1970.
  
  
\bibitem{F} {\sc T. Figiel}, {\em  Singular integral operators: a martingale approach}, Geometry of Banach spaces  (Strobl,
1989), London Math. Soc. Lecture Note Ser., vol. 158, Cambridge Univ. Press, Cambridge, 1990, pp. 95--110. 

\bibitem{FP} {\sc  R.  Fefferman, J. Pipher} {\em  Multiparameter operators and sharp weighted inequalities,}  Amer. J. Math. 119 (1997), no. 2, 337Ð369.

\bibitem{Grafakos_CFA-book_2008}
{\sc L.~Grafakos}, \emph{Classical {F}ourier analysis}, second ed., Graduate
  Texts in Mathematics, vol. 249, Springer, New York, 2008.


\bibitem{H} {\sc T.~Hyt\"{o}nen}, {\em The sharp weighted bound for general Calder\'on-Zygmund operators},  arXiv:1007.4330.

\bibitem{H1} {\sc T.~Hyt\"{o}nen}, {\em The vector-valued nonhomogeneous $Tb$ theorem},  arXiv:0809.3097.

\bibitem{HLRSUTV} {\sc T.~Hyt\"{o}nen, M. Lacey, M. C. Reguera, E. Sawyer, I. Uriarte-Tuero,  A. Vagharshakyan}, \emph{Weak and Strong type $ A_p$ Estimates for \cz Operators},
arXiv:1006.2530.


\bibitem{HMW} {\sc R. Hunt, B. Muckenhoupt, R. Wheeden}, {\em Weighted norm inequalities  for the conjugate function and the
Hilbert transform}, Trans. Amer. Math. Soc., {\bf 176} (1973), pp. 227-251.



\bibitem{LOP}{\sc A. Lerner, S. Ombrosi, C. P\'erez}, {\em A1 bounds for \cz operators related to a
problem of Muckenhoupt and Wheeden}, Math. Res. Lett., 16 (2009) 
no. 1, 149-156.

\bibitem{LPR}{\sc M. Lacey, S. Petermichl, M. Reguera}, {\em Sharp ${A}_2$ inequality for {H}aar shift operators}, Math. Ann., {\bf 348} (2010), 127--141.


\bibitem{Ler1} {\sc A. Lerner}, {\em A pointwise estimate for local sharp maximal
                  function with applications to singular integrals},  preprint, 2009.
                  
\bibitem{NT}
{\sc F.~L. Nazarov and S.~R. Treil}, \emph{The hunt for a {B}ellman
   function: applications to estimates for singular integral operators
and to
   other classical problems of harmonic analysis}, Algebra i Analiz
\textbf{8}
   (1996), no.~5, 32--162.
   
\bibitem{NTV1} {\sc F.~Nazarov, S.~Treil, and A.~Volberg,} {\em Cauchy
Integral and Calder\'{o}n-Zygmund  operators on nonhomogeneous 
spaces}, International Math. Research Notices, {\bf 1997}, No. 15, 103--726.


\bibitem{NTV2} {\sc F.~Nazarov, S.~Treil, and A.~Volberg,} {\em Weak type
estimates and  Cotlar inequalities for  Calder\'{o}n-Zygmund
operators on nonhomogeneous spaces}, International Math. Research
Notices, {\bf 1998}, No. 9, p.~463--487.

\bibitem{NTV3} {\sc F.~Nazarov, S.~Treil, and A.~Volberg,} {\em Accretive system $Tb$ theorems
on nonhomogeneous spaces}, Duke Math. J., {\bf 113} (2002), no. 2,  259--312.

\bibitem{NTV4}  {\sc F.~Nazarov, S.~Treil, and A.~Volberg,} {\em Nonhomogeneous 
$Tb$ theorem which proves Vitushkin's conjecture}, Preprint No. 519, CRM, Barcelona, 2002, 1-84.


\bibitem{NTV5} {\sc F.~Nazarov, S.~Treil, and A.~Volberg,} {\em The $Tb$ theorem on non-homogeneous spaces},  Acta Math., {\bf 190}  (2003), 151--239.

\bibitem{NTV6} {\sc F.~Nazarov, S.~Treil, and A.~Volberg,} {\em Two weight inequalities for individual Haar multipliers and other well localized operators}, 
 Math. Res. Lett. 15 (2008), no. 3, 583--597.
 
\bibitem{NTVlost} {\sc F.~Nazarov, S.~Treil, and A.~Volberg,}  {\em Two weight estimate for the Hilbert transform and corona decomposition for non-doubling measures}, Preprint 2005, 1-33. Put into arXive in 2010.

\bibitem{NTV7} {\sc F.~Nazarov, S.~Treil, and A.~Volberg,} 
{\em Two weight $T1$ theorem for the Hilbert transform: the case of doubling measures}, Preprint 2004, 1--40.

\bibitem{NTV-name} {\sc F.~Nazarov, S.~Treil and  A.~Volberg}, 
\emph{Bellman function in stochastic control and harmonic analysis.} 
Systems, approximation, singular integral operators, and related topics 
(Bordeaux, 2000),  393--423, Oper. Theory Adv. Appl., \textbf{129}, 
Birkh\"{a}user, Basel, 2001.

\bibitem{NTV-2w} {\sc F.~Nazarov, S.~Treil, and A.~Volberg,} {\em The Bellman functions and two-weight inequalities for 
Haar multipliers}, J. of Amer. Math. Soc., {\bf 12}, (1999), no. 4, 909-928.

\bibitem{Petm} {\sc S. Petermichl}, {\em Dyadic shifts and a logarithmic estimate for Hankel operators with matrix symbol}, 
C. R. Acad. Sci. Paris, S\'er. I Math., {\bf 330}, (2000), no. 6, 455-460.

\bibitem{PetmV} {\sc S. Petermichl, A. Volberg}, {\em Heating of the Ahlfors-Beurling operator: weakly quasiregular maps on the
plane are quasiregular}, Duke Math. J., {\bf 112} (2002), no. 2, 281-305.

\bibitem{Petm1} {\sc S. Petermichl}, {\em The sharp bound for the Hilbert transform on weighted Lebesgue spaces in terms of the classical $A_p$ characteristic. }Amer. J. Math. 129 (2007), no. 5, 1355--1375. 

\bibitem{Petm2} {\sc S. Petermichl}, {\em The sharp weighted bound for the Riesz transforms.} Proc. Amer. Math. Soc. 136 (2008), no. 4, 1237--1249.

\bibitem{Petm3} {\sc S. Petermichl}, {\em Dyadic shifts and a logarithmic estimate for Hankel operators with matrix symbol.} C. R. Acad. Sci. Paris S\'er. I Math. 330 (2000), no. 6, 455--460. 

\bibitem{PTV1}  {\sc C. P\'erez, S. Treil, A. Volberg}, {\em On $A_2$ conjecture and corona decomposition of weights}, arxiv1005.2630.


\bibitem{Saw1} {\sc E. Sawyer}, {\em A characterization of a two-weight norm inequality for maximal operators}, Studia Math.,
{\bf 75} (1982), no. 1, pp. 1-11.

\bibitem{Saw2} {\sc E. Sawyer}, {\em Two weight norm inequalities for certain maximal and integral operators}, Lecture Notes in Math.,  {\bf 908} (1982), 102--127.


\bibitem{To1}{\sc X. Tolsa}, {\em  $L^2$ boundedness for the Cauchy linear operator for continuous measures},  Duke Math. J., {\bf 98} (1999), no. 2, 269--304.


\bibitem{Wav_PastFuture}
{\sc S.~Treil and A.~Volberg}, \emph{Wavelets and the angle between past and future},
  J. Funct. Anal. \textbf{143} (1997), no.~2, 269--308.

\bibitem{V}{\sc A. Volberg}, {\em Matrix $A_p$ weights via $S$-function},  J.  Amer. Math. Soc., {\bf 10} (1997), no. 2, 445--466.



\bibitem{VolLip}{\sc A. Volberg}, {\em  Calder\'on--Zygmund capacities and operators on nonhomogeneous spaces},  CBMS Lecture Notes, Amer. Math. Soc., {\bf 100} (2003), pp. 1--167.

\bibitem{Vagh_HaarShift_2009}
{\sc Armen Vagharshakyan}, \emph{Recovering singular integral kernels from Haar
  shifts}, Proc. Amer. Math. Soc., {\bf 138} (2010), 4303--4309.

\bibitem{Wit}{\sc J. Wittwer}, {\em A sharp estimate on the norm of the martingale transform.} Math. Res. Lett. 7 (2000), no. 1, 1--12.

\bibitem{X} {\sc Y. Q. Xiang,} {\em Fast algorithms for 
Calder—n-Zygmund singular integral operators,} 
Appl. Comput. Harmon.Anal. 3 (1996), no. 2, 120Ð126.

\end{thebibliography}
\end{document}